\documentclass[12pt]{article}

\usepackage{fullpage,amsfonts,amsmath,amsthm,amssymb,mathtools,mathrsfs,graphicx,upgreek}
\usepackage{verbatim,url}
\usepackage{graphicx,hyperref,enumitem}
\usepackage[noadjust,sort]{cite}
\usepackage{tocloft}

\usepackage{authblk}

\usepackage[normalem]{ulem}

\usepackage[top=2.54cm, bottom=2.54cm, left=2.54 cm, right=2.54 cm]{geometry}
\newcommand{\lab}[1]{\label{#1}}                

\usepackage[usenames,dvipsnames]{color}

\newcommand{\remove}[1]{}
\newcommand\eqn[1]{(\ref{#1})}

\newcommand{\be}{\begin{equation}}
\newcommand{\bel}[1]{\begin{equation}\lab{#1}\ }
\newcommand{\ee}{\end{equation}}
\newcommand{\bea}{\begin{eqnarray}}
\newcommand{\eea}{\end{eqnarray}}
\newcommand{\bean}{\begin{eqnarray*}}
\newcommand{\eean}{\end{eqnarray*}}

\newtheorem{thm}{Theorem}
\newtheorem{cor}[thm]{Corollary}

\newtheorem{lemma}[thm]{Lemma}
\newtheorem{definition}[thm]{Definition}
\newtheorem{claim}[thm]{Claim}

\newtheorem{observation}[thm]{Observation}

\newcommand{\floor}[1]{\left\lfloor#1\right\rfloor}

\newcommand{\bm}[1]{{\pmb #1}}
\def\time{{\tau_{{\tt HAM}}}}
\def\cyclic{{\tt C}_{\tt cyclic}}

\def\G{{\mathcal G}}

\def\P{{\mathcal P}}

\newcommand{\bbN}{\mathbb{N}}

\newcommand{\scr}{\mathcal}

\def\ex{{\mathbb E}}
\def\pr{{\mathbb P}}

\def\bbN{{\mathbb N}}

\def\eps{\epsilon}

\date{}

\title{Hamilton Cycles in the Semi-random Graph Process}
\author{Pu Gao\thanks{The first, the third, and the last authors are partially supported by NSERC.} \\
\small{Department of Combinatorics and Optimization\\
University of Waterloo, Waterloo, Canada \\
\texttt{pu.gao@uwaterloo.ca}
}
\and
Bogumi\l{} Kami\'{n}ski \\ 
\small{
Decision Analysis and Support Unit \\
SGH Warsaw School of Economics, Warsaw, Poland\\
\texttt{bkamins@sgh.waw.pl}
}
\and
Calum MacRury$^*$ \\
\small{Department of Computer Science\\ University of Toronto, Toronto, Canada\\
\texttt{calum.macrury@gmail.com}
}
\and
Pawe\l{} Pra\l{}at$^*$ \\ 
\small{Department of Mathematics\\ Ryerson University, Toronto, Canada\\
\texttt{pralat@ryerson.ca}
}
}

\begin{document}
\maketitle

\begin{abstract}
The semi-random graph process is a single player game in which the
player is initially presented an empty graph on $n$ vertices. In each
round, a vertex $u$ is presented to the player independently and uniformly
at random. The player then adaptively selects a vertex $v$, and adds
the edge $uv$ to the graph. For a fixed monotone graph property, 
the objective of the player is to force the graph to satisfy this property
with high probability in as few rounds as possible.

We focus on the problem of constructing a Hamilton cycle in as few rounds
as possible. In particular, we present a novel strategy for the player which achieves
a Hamiltonian cycle in $(2+4e^{-2}+0.07+o(1)) \, n < 2.61135 \, n$ rounds, assuming that a specific non-convex optimization problem has
a negative solution (a premise we numerically support).
Assuming that this technical condition holds, this improves upon the previously
best known upper bound of $3 \, n$ rounds. We also show that
the previously best lower bound of $(\ln 2 + \ln (1+\ln 2) + o(1)) \, n$ is not tight.
\end{abstract}

\section{Introduction}

In this paper, we consider the \textbf{semi-random process} introduced recently in~\cite{process1} that can be viewed as a ``one player game''. The process starts from $G_0$, the empty graph on the vertex set $[n]=\{1,\ldots,n\}$. In each step $t$, a vertex $u_t$ is chosen uniformly at random from $[n]$. Then, the player (who is aware of graph $G_t$ and vertex $u_t$) needs to select a vertex $v_t$ and add an edge $u_tv_t$ to $G_t$ to form $G_{t+1}$. The goal of the player is to build a (multi)graph satisfying a given monotonely increasing property ${\mathcal A}$ as quickly as possible.

A \textbf{strategy} in this context is a sequence of functions $f_1,f_2,\ldots$, where for each $t \in \bbN$, $f_t(u_1,v_1,\ldots, u_{t-1},v_{t-1},u_t)$ is a distribution over $[n]$, given the history of the process up to and including step $t-1$, and vertex $u_t$. Then $v_t$ is chosen according to this distribution. If $f_t$ is an atomic distribution, then $v_t$ is determined by $u_1,v_1,\ldots,u_{t-1},v_{t-1},u_t$. As $f_t$ is determined by $u_t$, and the history of the process up to step $t-1$, it means that the player needs to select her strategy in advance, before the game actually starts. Given ${\bf f}=(f_1,f_2,\ldots)$ and real number $0<q<1$, let $\tau_q({\bf f},n)$ be the minimum $t$ for which $\pr(G_t\in {\cal A})\ge q$, where $(G_i)_{i=0}^t$ is obtained following strategy ${\bf f}$. Define
\[
\tau_{q}({\cal A},n) = \min_{{\bf f}} \tau_q({\bf f},n),
\]
where the minimum is over all strategies. 
Let 
\[
\tau_{{\cal A}}=\lim_{q\to 1-}\limsup_{n\to\infty} \frac{\tau_{q}({\cal A},n)}{n};
\]
note that the limit above exists, since for every $n$ the function $q\to \frac{\tau_{q}({\cal A},n)}{n}$ is nondecreasing, as $\cal A$ is  monotonely increasing. 

\medskip

In this paper, we concentrate on property ${\mathcal A = {\tt HAM}}$ that a graph has a Hamilton cycle. 
It was observed in~\cite{process1} that 
$$
1.21973  \le \ln 2+\ln(1+\ln2)  \le \time \le 3.
$$

We improve both upper and lower bounds for $\time$. For the upper bound, we need to assume some ``technical condition'' $\P$ that claims that some function is negative on its domain---the function is defined in Subsection~\ref{sec:function}. Unfortunately, we could not prove that $\P$ holds but, in Subsection~\ref{sec:numerical}, we provide strong numerical evidence for it. For the lower bound, we do not optimize the argument (as it gives a small improvement anyway) but aim for a relatively easy proof that shows that the currently existing bound is not tight.

\medskip

Here are our main results. Theorem~\ref{thm:upper_bound} is proved in Section~\ref{sec:upper_bound} whereas the proof of Theorem~\ref{thm:lower_bound} can be found in Section~\ref{sec:lower_bound}.

\begin{thm}\label{thm:upper_bound}
Suppose that property $\P$ holds. Then, 
$$
\time \le 2+4e^{-2}+0.07 < 2.61135.
$$
\end{thm}

\begin{thm}\label{thm:lower_bound}
There exists a universal constant $\eps > 10^{-8}$ such that
$$
\time \ge \ln 2+\ln(1+\ln2) + \eps.
$$
\end{thm}

All asymptotics in this paper refer to $n \to \infty$. We say that an event holds \textbf{asymptotically almost surely} (\textbf{a.a.s.}) if the probability that it holds tends to 1 as $n \to \infty$. In the proofs we use the standard Landau notation. Given two sequences of real numbers $a_n$ and $b_n$, we write $a_n=O(b_n)$ if there exists a constant $C>0$ such that $|a_n|\le C|b_n|$ for all $n$. We write $a_n=o(b_n)$ if $b_n>0$ for all sufficiently large $n$ and $\lim_{n\to\infty} a_n/b_n=0$.

\section{Upper Bound}\label{sec:upper_bound} 

In order to obtain an upper bound for $\time$, one needs to propose a strategy for the player to build a graph during the semi-random process, and show that after a certain number of steps the resulting graph is a.a.s.\ Hamiltonian. 

In order to warm up, let us recall an observation made in~\cite{process1} that gives $\time \le 3n$ a.a.s. To see it, the following simple strategy can applied: let $v_t= (t-1) \pmod n +1 $ for all $1\le t\le 3n$. Note that this is a non-adaptive strategy, that is, function $f_t$ does not depend on the history of the process nor vertex $u_t$ chosen at time $t$. More importantly, it is easy to see that the resulting graph has the same distribution as the well-known $G_{\text{3-out}}$ process that is Hamiltonian a.a.s.~\cite{3-out}. 

In general, the \textbf{$m$-out process} is defined for any natural number $m$: each vertex $v \in [n]$ independently chooses $m$ random out-neighbours from $[n]$ to create the random digraph $D_{m\text{–out}}$. We then obtain $G_{m\text{–out}}$ by ignoring orientations. Note that $G_{m\text{–out}}$ is a multi-graph (it may have loops or multiple edges) with minimum degree $m$ and precisely $mn$ edges. In the model, we can either allow these multiple edges and loops, replace multiple edges with single edges and remove loops, or condition on them not occurring.  (Since the probability that there are no multiple edges is bounded away from zero, any property that holds a.a.s.\ in the model that allows multiple edges also holds a.a.s.\ when we condition on no multiple edges.) For our application, when the strategy creates a multiple edge or a loop in the underlying undirected graph, we simply ``discard'' that edge. That is, we will not use that edge for the construction of a Hamilton cycle.

This section is structured as follows. In Subsection~\ref{sec:upper_strategy}, we define a strategy for the player that creates a random graph $G^*$. The main goal is to prove that $G^*$, together with some additional $o(n)$ semi-random edges, is Hamiltonian a.a.s. Since the argument is quite involved, an overview of the proof is provided in Subsection~\ref{sec:upper_overview}. In Subsection~\ref{sec:upper_notation}, we introduce definitions and notation that will be used through the entire paper. Some useful properties of the graphs involved in the argument are extracted and proved in Subsection~\ref{sec:upper_properties}. In order to achieve our goal, in particular, we need to prove that a.a.s.\ $G^*$ has a 2-matching with $o(n)$ components (the definition is provided in Subsection~\ref{sec:upper_overview}). Subsection~\ref{sec:upper_2-matching_preparation} prepares us for this task. As already mentioned earlier, this part requires property $\P$ that is defined at the end of Subsection~\ref{sec:function}. The proof that a.a.s.\ $G^*$ has a 2-matching with few components is finished in Subsection~\ref{sec:upper_2matching}. Now, it is enough to guide the semi-random process such that after additional $o(n)$ rounds the graph has a Hamiltonian cycle. This last task does not depend on the argument used to show that $G^*$ has the desired 2-matching and so, in fact, we do it earlier, in Subsection~\ref{sec:upper_final}.

\subsection{Our Strategy}\label{sec:upper_strategy}

In this subsection, we define a strategy for the player that creates a random (multi)graph $G^*$. 
It will be convenient to work with the directed graph $D_t$ underlying $G_t$. For each edge $u_tv_t$ that is added to $G_t$ at time $t$, we put a directed edge from $v_t$ to $u_t$ in $D_t$.
As mentioned before, for the construction of a Hamilton cycle we will only use edges from a subgraph of $G_t$. For any multigraph $G$, let $\hat G$ denote the simple graph obtained from $G$ by deleting all loops and replacing all multiple edges by single edges. Thus, the sequence of multigraphs $(G_t)$ immediately yields the corresponding sequence of simple graphs $(\hat G_t)$.

Consider the following strategy that will be defined in four phases.  During the first phase, for $1\le t\le 2n$, let $v_t= (t-1) \pmod n +1$. It is clear that $G_{2n}$ has the same distribution as $G_{\text{2-out}}$. Let $V_0$ and $V_1$ be the sets of vertices in $D_{2n}$ of in-degree 0 and, respectively, of in-degree 1. During the second phase, two out edges are added from every vertex in $V_0$ and one out edge is added from every vertex in $V_1$. During the third phase, we add $0.07n$ directed edges uniformly at random, that is, in each step, $v_t$ is uniformly chosen from $[n]\setminus \{u_t\}$. We call $v$ a {\em deficit vertex} if after phase 3 its degree is less than 4. Then, in the fourth and last phase, we repeatedly add a semi-random edge, coloured golden for convenience, coming out of a deficit vertex until its degree in the current underlying undirected {\em simple} graph (that is, in $\hat G_t$) becomes at least 4. Let $\tau_i$ denote the last step of phase $i$ (in particular, $\tau_1=2n$). 
Note that a golden semi-random edge is added out of $u$ only if a loop or a multiple edge incident with $u$ was created in the process $(G_t)$ during the first two phases. It is easy to show, by a standard first moment calculation, that $G_{\tau_3}$ has $O(1)$ loops or parallel edges, and $o(1)$ other types of multiple edges in expectation. If $v$ is incident with a loop in $G_{\tau_3}$ then $v$ may send out up to two semi-random edges in phase 4. If $v$ is incident with a parallel edge in $G_{\tau_3}$ then $v$ may send out at most one semi-random edge in the final phase.  Hence, a.a.s.\ $G_{\tau_4}$ and $\hat G_{\tau_4}$ have the following property. 
\begin{tabbing}
{(\tt E)}: \hspace{0.2cm} \= There are at most $\ln\ln n$ double edge or loops in $G_{\tau_4}$ and they are all vertex disjoint.
\\
\>There are at most $\ln\ln n$ golden edges, inducing vertex-disjoint paths of length 1 or 2,
\\
\>
and every pair of deficit vertices are at distance at least $\ln n/5$ from each other  in $\hat G_{\tau_4}$.
\end{tabbing}
Thus a.a.s.\ the total number of semi-random edges added during the last two phases is $(0.07+o(1))n$. Note that the addition of the golden edges guarantees that the minimum degree of $\hat G_{\tau_4}$ is at least 4, which will be used in the proof later. Finally, let $G^*=G_{\tau_4}$, the multigraph obtained after the last step of phase 4. As we will only use edges in $\hat G_{\tau_4}\subseteq G^*$, we will mainly focus on the process $(\hat G_t)$.


\subsection{Overview of the Proof}\label{sec:upper_overview}

Let us present an overview of the proof of Theorem~\ref{thm:upper_bound}. First, we will investigate how long it takes to construct graph $G^*$.
\begin{lemma}\label{lem:number_of_edges}
A.a.s.\ the following holds
$$
|E(G^*)| = (2+4e^{-2}+0.07+o(1))n.
$$
\end{lemma}

In order to state the next lemma, we need one definition. A  \textbf{2-matching} in a graph $G$ is a simple subgraph $H$ of $G$ with maximum degree at most 2, that is, a collection of vertex-disjoint paths and cycles. Moreover, we assume that $V(H)=V(G)$ so some paths in $H$ could be isolated vertices. 

\medskip

Now, we are ready to state the lemma. This is the place where property $\P$ is needed.

\begin{lemma}\label{lem:2-matching}
Suppose that property $\P$ holds. Then, a.a.s.\ $G^*$ has a 2-matching with $o(n)$ components.
\end{lemma}

The final ingredient does not require property $\P$ anymore. 

\begin{lemma} \lab{lem:Ham}
Suppose $G^*$ has a 2-matching with $o(n)$ components. Then, there exists an adaptive strategy such that a.a.s.\ the semi-random process builds a Hamiltonian graph within an additional $o(n)$ steps. 
\end{lemma}

Theorem~\ref{thm:upper_bound} follows immediately from the above three lemmas. 
Our strategy for constructing a Hamilton cycle in Lemma~\ref{lem:Ham} is the same as that in~\cite{3-out} where a Hamilton cycle is found in $G_{3\text{-out}}$. We start with a 2-matching $F$ of $G^*$ which has $o(n)$ components. Then, we take an arbitrary component $C$ of $F$ and let $P$ be a path that spans all vertices of $C$. By applying Pos\'{a} rotations, we use either edges in $G^*$, or additional $o(n)$ edges added to $G^*$ to repeatedly absorb vertices in other components of $F$ into the long path we carefully construct, until finally completing the path into a Hamilton cycle. Having less available edges in $G^*$ than in $G_{3\text{-out}}$ requires some new treatments in the proof of Lemma~\ref{lem:Ham}.

In order to prove Lemma~\ref{lem:2-matching}, as it is done in~\cite{3-out}, we will apply Tutte and Berge's formula for the size of a maximum 2-matching of $\hat G_{\tau_4}\subseteq G^*$. However, as we have significantly less edges in $\hat G_{\tau_4}$ than in $G_{3\text{-out}}$, it becomes much more challenging to verify the Tutte-Berge conditions. Rough bounds that worked in~\cite{3-out} fail to work in our setting and, in order to achieve a tighter bound we end up with an optimization problem involving a high dimensional objective function. That results in the technical property $\cal P$ that we only support numerically.

\subsection{Definitions and Notation}\label{sec:upper_notation}

In this subsection, we introduce basic definitions and notation that will be used throughout the paper. Let us start from graph theoretic ones.
For a given subset of vertices $S \subseteq V(G)$, let $G[S]$ be the \textbf{graph induced by set $S$}, that is, $V(G[S])=S$ and 
$$
E(G[S]) = \{ uw \in E(G) : u,v \in S\} \subseteq E(G). 
$$
Let $e(S)$ denote the number of edges induced by set $S$, that is,
$$
e(S) = |E(G[S])| = | \{ xy \in E(G) : x,y \in S \} |.
$$
Moreover, let 
$$
N(S) = \{v\in V(G) \setminus S : \exists u\in S \text{ such that } uv \in E(G) \}. 
$$
Finally, we say that $S$ is an \textbf{independent set} if $S$ induces no edge, that is, $e(S) = 0$. 

Given subsets of vertices $U,W \subseteq V(G)$, 
let $e(U,W)$ denote the number of edges with exactly one end in $U$ and the other end in $W$, that is,  
$$
e(U,W) = | \{ uw \in E(G) : u \in U, w \in W \} |.
$$
For a given vertex $v \in V(G)$, let $\deg(v)$ be the \textbf{degree of $v$}, that is, the number of neighbours of $v$ in $G$. Let $\delta(G) = \min\{ \deg(v) : v \in V(G)\}$ denote the \textbf{minimum degree} of a graph $G$. 

\medskip

For a directed graph $D$ and a given vertex $v \in V(D)$, let $\deg^-(v)$ and $\deg^+(v)$ be the in- and out-degree of $v$, that is, the number of directed edges going to $v$ and, respectively, going from $v$ in $D$.

\medskip

For sequences of real numbers $a_n$ and $b_n$, we say $a_n=poly(n)$ if there exists a constant $C>0$ such that $n^{-C}<a_n<n^C$ for every $n$.  We say $a_n=O(b_n)$ if there exists a constant $C>0$ such that $|a_n|<C|b_n|$ for all $n$. We say $a_n=o(b_n)$ if eventually $b_n>0$ and $\lim_{n\to\infty} a_n/b_n=0$.

\medskip

Finally, let us introduce the \textbf{binomial random graph} $\G(n,p)$. More precisely, $\G(n,p)$ is a distribution over the class of graphs with vertex set $[n]$ in which every pair $\{i,j\} \in \binom{[n]}{2}$ appears independently as an edge in $G$ with probability~$p$. Note that $p=p(n)$ may (and in our application it does) tend to zero as $n$ tends to infinity. 

\subsection{Some Technical Properties and Proof of Lemma~\ref{lem:number_of_edges}}\label{sec:upper_properties}

Let us start with the following simple observations. 

\begin{observation}\lab{obs}
Our process can be coupled such that the following properties hold.
\begin{enumerate}
    \item [(a)] $G_{2n}$ has the same distribution as $G_{\text{2-out}}$ and thus $\hat G_{\tau_1}\subseteq G_{\text{2-out}}$. 
    \item [(b)] $\hat G_{\tau_2}$ is a subgraph of $G_{4\text{-out}}$; in particular, 
    \begin{equation}\label{eq:prob_tildeG}
        \pr \left( S \subseteq E(\hat G_{\tau_2}) \right) \le \left( \frac {8}{n} \right)^{|S|}, \text{ for any } S \subseteq {[n] \choose 2}.
    \end{equation}
    \item [(c)] $\delta(\hat G_{\tau_4}) \ge 4$ and
    \begin{eqnarray}
        \pr \Big( S \subseteq E(\hat G_{\tau_3})\Big) &\le& \left( \frac {8.15}{n} \right)^{|S|}, \text{ for any } S \subseteq {[n] \choose 2} \label{eq:prob_Gstar}\\
        \pr \Big( S \subseteq E(\hat G_{\tau_4})\Big) &\le& \left( \frac {13}{n} \right)^{|S|}, \text{ for any } S \subseteq {[n] \choose 2}\label{eq:prob_Gfinal}.
    \end{eqnarray}
    
\end{enumerate}
\end{observation}
\begin{proof}
\emph{Part (a)}: The property follows immediately from the construction of our process.

\medskip

\emph{Part (b)}: Recall that $G_{\tau_2}$ is constructed from $G_{2\text{-out}}$ by adding two out edges from every vertex in $V_0$ and one out edge from every vertex in $V_1$. If, instead, two out edges are added from \emph{every} vertex in $G_{2\text{-out}}$, we would get a graph with the same distribution as $G_{4\text{-out}}$. Hence, one may easily couple our process such that $G_{\tau_2}$ is a subgraph of $G_{4\text{-out}}$. 
In order to see that~(\ref{eq:prob_tildeG}) holds, note first that 
$$
\pr \Big( e \in E(\hat G_{\tau_2}) \Big)=\pr \Big( e \in E(G_{\tau_2}) \Big) \le \pr \Big( e \in E(G_{4\text{-out}}) \Big) = 1 - \left( 1 - \frac {1}{n} \right)^8 \le \frac {8}{n}.
$$
The desired inequality holds after observing that $S' \subseteq E(\hat G_{\tau_2})$ does not increase the probability that an edge $e \notin S'$ is also in $E(\hat  G_{\tau_2})$.


\medskip

\emph{Part (c)}: The fact that $\delta(\hat G_{\tau_4}) \ge 4$ follows immediately by construction of $\hat G_{\tau_4}$.  For~\eqn{eq:prob_Gstar}, we note that part (b) implies that our process can be coupled such that $\hat G_{\tau_2}$ is a subgraph of $G_{4\text{-out}}$.  As a result, $\hat G_{\tau_3}$ can be viewed as a subgraph of $G_{4\text{-out}} \cup \G(n,0.07n)$.
Thus, by the union bound we get that
\[
\pr(e\in E(\hat G_{\tau_3})) \le \pr(e\in E(\hat G_{\tau_2})) +\pr(e\in E(\G(n,0.07n)))\le \frac{8}{n}+\frac{0.14+o(1)}{n}<\frac{8.15}{n}.
\]
The assertion follows by noting that the presence of other edges do not increase the probability that $e\in E(\hat G_{\tau_3})$.

In order to see that~(\ref{eq:prob_Gfinal}) holds, we apply the same argument after noting that every vertex sends out at most two golden semi-random edges. As a result, $\hat G_{\tau_4}$ can be viewed as a subgraph of $G_{6\text{-out}} \cup \G(n,0.07n)$.
\end{proof}

The next lemma collects some important properties of the graphs involved in the process that will be used in various places of this paper. In particular, part~(a) immediately implies Lemma~\ref{lem:number_of_edges}.

\begin{lemma} \lab{lem:aas} A.a.s.\ the following properties hold.
\begin{enumerate}
\item[(a)] $D_{\tau_1}$ has asymptotically $e^{-2}n$ vertices of in-degree 0 and $2e^{-2}n$ vertices of in-degree 1. In other words, $|V_0|=(e^{-2}+o(1))n$ and $|V_1|=(2e^{-2}+o(1))n$.
\item[(b)] For every $\eps>0$ there exists $\delta=\delta(\eps)>0$ such that for all $S \subseteq [n]$ with $|S|\le \delta n$, $S$ induces at most $(1+\eps)|S|$ edges in $\hat G_{\tau_4}$. 
\item[(c)] All $S\subseteq [n]$ with $|S|\le 0.005n$ induce at most $1.9|S|$ edges in $\hat G_{\tau_4}$. 
\end{enumerate}
\end{lemma}

\begin{proof}
\emph{Part (a)}: Let $v \in [n]$ be any vertex of $D_{2n}=D_{\tau_1}$. Clearly,
\begin{eqnarray*}
\pr ( \deg^-(v) = 0 ) &=& \left( 1 - \frac {1}{n} \right)^{2n} = e^{-2} + o(1) \\
\pr ( \deg^-(v) = 1 ) &=& (2n) \cdot \frac {1}{n} \cdot \left( 1 - \frac {1}{n} \right)^{2n-1} = 2e^{-2} + o(1).
\end{eqnarray*}
It follows that $\ex (|V_0|) = (e^{-2}+o(1)) n$ and $\ex (|V_1|) = (2e^{-2}+o(1)) n$. It is straightforward to show the concentration for these random variables (for example, by using the second moment method; we omit details) and so part (a) holds.

\medskip

\emph{Part (b)}: 
Let us fix $\eps>0$ and $s = s(n) \in \bbN$. By~(\ref{eq:prob_Gfinal}), the expected number of sets $S \subseteq [n]$ with $|S|=s$ that induce at least $(1+\eps)s$ edges in $\hat G_{\tau_4}$ is at most
\begin{eqnarray*}
g(s) &:=& {n \choose s} {{s \choose 2} \choose (1+\eps)s} \left( \frac {13}{n} \right)^{(1+\eps)s} \le \left( \frac {en}{s} \right)^s \left( \frac {es^2/2}{(1+\eps)s} \right)^{(1+\eps)s} \left( \frac {13}{n} \right)^{(1+\eps)s} \\
&=& \left( \frac {e^{2+\eps} 6.5^{1+\eps}}{(1+\eps)^{1+\eps}} \left( \frac {s}{n} \right)^\eps \right)^s  \le \left( 6.5e^2 \left( \frac {6.5es}{n} \right)^\eps \right)^s.
\end{eqnarray*}
Clearly, 
$$
g(s) \le \left( 6.5e^2 \left( 6.5e \delta \right)^\eps \right)^s \le (1/2)^s,
$$
provided that $s \le \delta n$ and $\delta = \delta(\eps) > 0$ is sufficiently small (the optimal value of $\delta$ is $(13e^2)^{-1/\eps}/(6.5e)$). On the other hand, if (for example) $s \le \ln n$, then $g(s) \le n^{-\eps s/2} \le n^{-\eps/2}$. It follows that the expected number of sets $S \subseteq [n]$ with $|S| \le \delta n$ that induce at least $(1+\eps)|S|$ edges is at most
$$
\sum_{s =1}^{\delta n} g(s) \le \sum_{s=1}^{\ln n} n^{-\eps/2} + \sum_{s=\ln n}^{\delta n} (1/2)^s \le (\ln n) n^{-\eps/2} + 2 (1/2)^{\ln n} = o(1).
$$
Part (b) holds by Markov's inequality. 




\medskip

\emph{Part (c)}:  For a given $s=s(n) \in \bbN$, let $X_s$ be number of sets $S \subseteq [n]$ with $|S|=s$ that induce at least $1.9s$ edges in $\hat G_{\tau_4}$, and let $Y_s$ be the number of sets $S \subseteq [n]$ with $|S|=s$ that induce at least $t(s)$ edges in $\hat G_{\tau_3}$, where
\[
t(s)=\left\{
\begin{array}{ll}
1.2s & \mbox{if $s\le \ln n$}\\
1.89s & \mbox{if $s>\ln n$}.
\end{array}
\right.
\]
As a.a.s.\ $\hat G_{\tau_4}\in {\tt E}$, it follows that a.a.s.\  $X_s\le Y_s$ for all $s$, since the number of golden edges induced by $S$ is at most  $\min\{2s/3, \ln\ln n\} \le \min\{0.7s,\ln\ln n\}$, given $\hat G_{\tau_4}\in {\tt E}$, and $1.9s-\ln\ln n\ge 1.89s$ when $s>\ln n$.  Let $g(s)=\ex(Y_s)$. By~(\ref{eq:prob_Gstar}), we get that
\begin{eqnarray*}
g(s) &\le& {n \choose s} {{s \choose 2} \choose t(s)} \left( \frac {8.15}{n} \right)^{t(s)} \le \left( \frac {en}{s} \right)^s \left( \frac {es^2/2}{t(s)} \right)^{t(s)}  \left( \frac {8.15}{n} \right)^{t(s)}.
\end{eqnarray*}
If $s \le \ln n$, then $g(s) \le n^{-0.8}$. On the other hand, if $\ln n<s \le \ \delta n$ with $\delta = 0.005$, then 
$$
g(s) \le \left( e \left( \frac {8.15e}{3.78} \right)^{1.89} \delta^{0.89} \right)^s < 0.7^s.
$$
It follows that the expected number of sets $S \subseteq [n]$ with $|S| \le \delta n$ that induce at least $1.9|S|$ edges in $\hat G_{\tau_4}$ is at most
$$
\sum_{s=1}^{\delta n} g(s)\le \sum_{s=1}^{\ln n} n^{-0.8} + \sum_{s=\ln n}^{\delta n} 0.7^s  = o(1).
$$
Part (c) holds by Markov's inequality. 
\end{proof}

\subsection{Proof of Lemma~\ref{lem:Ham}}\label{sec:upper_final}

The whole subsection is devoted to prove Lemma~\ref{lem:Ham}. In order to achieve it, we will use a powerful proof technique introduced by Pos\'{a} in~\cite{Posa}. Suppose that $F$ is a 2-matching (that is, a collection of vertex-disjoint paths and cycles) of $\hat G_{\tau_4}\subseteq G^*$ with $o(n)$ components. We will use Pos\'{a} rotations to extend a path in $\hat G_{\tau_4}$ to longer and longer paths, and eventually extend a Hamilton path to a Hamilton cycle, by adding $o(n)$ extra semi-random edges. During the process of extending the paths, we will use edges in $\hat G_{\tau_4}$ whenever possible. If no edges in $\hat G_{\tau_4}$ are of help, then we will use semi-random edges where we strategically choose $v_t$ to help us with the extension of the paths. 

We start from a path $P=u_1u_2\ldots u_h$ in $F$. If $F$ is a collection of cycles, then we arbitrarily take a cycle and let $P$ be the path obtained by deleting an arbitrary edge in that cycle. Given a path $P$ and an edge $u_h u_j$, $1<j<h-1$ we can create another path of length $h$, namely, $P'=u_1u_2\ldots u_j, u_h, u_{h-1}, \ldots, u_{j+1}$ with a new endpoint $u_{j+1}$. We call this operation a \textbf{Pos\'{a} rotation}. 
Let ${\cal S}$ be the set of paths in $\hat G_{\tau_4}$ on the same set of vertices as $P$ obtained by fixing $u_1$ and performing any sequence of Pos\'{a} rotations on $P$. Let ${\tt End}$ denote the union of the end vertices of paths in ${\cal S}$ other than $u_1$. 

\medskip

Let us independently consider the following two cases:

\medskip

\noindent {\em Case 1: there is $x\in {\tt End}$ and $y\notin V(P)$ such that $xy\in E(\hat G_{\tau_4})$.}  If $y$ is in a cycle $C$ in $F$, then we can extend $P$ to a longer path on $V(P)\cup V(C)$. On the other hand, if $y$ is in a path $P'$ in $F$, then without loss of generality we may assume that $P'=v_1v_2\ldots v_{\ell} \ldots, v_{h}$ with $v_{\ell}=y$ where $\ell> h/2$. We can now extend $P$ to a longer path on vertex set $V(P)\cup \{v_1,\ldots, v_{\ell}\}$. After that operation, the number of vertex-disjoint paths and cycles remains the same or decreases by one. 

\medskip

\noindent {\em Case 2: for every $x\in {\tt End}$, $N(x)\subseteq V(P)$.} Colour vertices in ${\tt End}$ blue or red as follows. If $u_i\in {\tt End}$ and none of the two neighbours of $u_i$ (or just one neighbour of $u_i$ if $i=h$) on $P$ are in ${\tt End}$, then colour $u_i$ red; otherwise, colour it blue. Let us start with the following observation about red vertices.

\begin{claim}
Let $U$ denote the set of red vertices in ${\tt End}$. Then $U$ induces an independent set in $\hat G_{\tau_4}$.
\end{claim}
\begin{proof}
For a contradiction, suppose that $x,y$ are both red vertices in ${\tt End}$ and $xy$ is an edge in $\hat G_{\tau_4}$. Without loss of generality, suppose that $y$ was added to ${\tt End}$ before $x$ and let $P'$ be the path obtained via Pos\'{a} rotation with $y$ being the other end. Let $x=u_i$. Since $x$ is red, neither $u_{i-1}$ nor $u_{i+1}$ is in ${\tt End}$. Thus, the two neighbours of $x$ on $P'$ must be $u_{i-1}$ and $u_{i+1}$. But then we can get another path on $V(P)$ via Pos\'{a} rotation on $P'$ where one of $u_{i-1}$ and $u_{i+1}$ becomes an end vertex. This contradicts with the fact that $x$ is red. It follows that $U$ must be an independent set in $G^*$.
\end{proof}

By the usual argument of Pos\'{a} rotation, for every $u_i\in N({\tt End})$, we must have 
$$
\{u_{i-1},u_{i+1}\}\cap {\tt End}\neq \emptyset. 
$$
In particular, it implies that $|N({\tt End})| < 2|{\tt End}|$. However, using the above claim, we get a slightly stronger bound. Let $x_1$ and $x_2$ be the number of red and, respectively, blue vertices in ${\tt End}$. Since the set of red vertices in ${\tt End}$ induces an independent set in $\hat G_{\tau_4}$, it follows that
\begin{equation}
|N({\tt End})|\le 2x_1+x_2-1 = |{\tt End}|+x_1-1. \lab{x12}
\end{equation}

Our next task and the main ingredient of the proof of the lemma is the next claim.

\begin{claim} \lab{claim:linear}
$|{\tt End}|=\Omega(n)$.
\end{claim}
\begin{proof}
In order to simplify the notation, let $S={\tt End}$.
Let $\eps_0>0$ be a sufficiently small constant that will be determined soon. We will show that $|S| \ge \eps_0 n$. For a contradiction, suppose that $|S|<\eps_0 n$, and let 
$$
{\cal N}_i=\{x\notin S:\ e(\{x\}, S)=i\}, \qquad \qquad
n_i = |{\cal N}_i|.
$$

\begin{claim} For every $0<\eps\le 1$, 
$\sum_{i\ge 1} n_i \ge (2-\eps) |S|$, provided $\eps_0=\eps_0(\eps)$ is sufficiently small.
\end{claim}

\noindent Indeed, by Lemma~\ref{lem:aas}(b) applied with $\eps'=\eps/2$ and $S \cup N(S)$, we get that a.a.s.\ 
$$
e\left( S \cup \bigcup_{i \ge 2} {\cal N}_i \right) \le (1+\eps') \left|S \cup \bigcup_{i \ge 2} {\cal N}_i \right|,
$$ 
provided $\eps_0$ is sufficiently small. It follows that
\begin{equation}
e(S) +\sum_{i\ge 2} i n_i  \le (1+\eps')\left(|S|+\sum_{i\ge 2} n_i\right).\label{e(S)}
\end{equation}
On the other hand, by Observation~\ref{obs}(c), $\delta(\hat G_{\tau_4}) \ge 4$ and so
$$
2e(S) + \sum_{i\ge 1} n_i \ge 4|S|.
$$
Substituting $
2e(S)\le (2+2\eps')|S|+\sum_{i\ge 2}(2+2\eps'-2i)n_i$ from~\eqn{e(S)}
into the above yields 
\[
(2-2\eps')|S|\le n_1+\sum_{i\ge 2}(2+2\eps'-2i+1) n_i
\]
By the definition of $\eps'$ and as $\eps'\le 1/2$,
we get
\begin{eqnarray}
(2-\eps)|S|&=&(2-2\eps')|S|\le n_1+\sum_{i\ge 2}(2+\eps-2i+1) n_i \label{SS}\\
&\le& n_1\le \sum_{i\ge 1} n_i. \lab{S}
\end{eqnarray}
This finishes the proof of the claim.

\medskip

We apply the above claim with $\eps=0.05$ so we may assume that 
\begin{equation}
|N(S)|\ge (2-\eps)|S|\quad \mbox{if $|S|\le \eps_0 n$.}\lab{N(S)}
\end{equation}
By~\eqn{x12} and~\eqn{N(S)},
\[
(2-\eps)|S|\le |S|+x_1-1,
\]
and hence
\[
(1-\eps)|S| \le x_1-1.
\]
Let $X_1$ denote the set of red vertices and let $X_2$ be the set of blue vertices in $S$.
Let $X_1'\subseteq X_1$ be the set of red vertices with at least 2 blue neighbours. 
\begin{claim}
$|X'_1|\le 1.3\eps|S|.$
\end{claim} 

\noindent Indeed, consider the subgraph of $G$ induced by $Y=X'_1\cup X_2$. By the definition of $X'_1$, $Y$ induces at least $2|X'_1|$ edges. Hence, $2|X'_1|\le 1.1 (|X'_1|+|X_2|)$. As $x_2<\eps|S|$, we have $|X_1'|\le (1.1/0.9)\eps|S|< 1.3 \eps |S|$, which finishes the proof of the claim.

\medskip

Therefore, every vertex in $X_1\setminus X_1'$ has at least 3 neighbours in $\overline{S}$. Thus, $e(S, \overline{S})\ge 3|X_1\setminus X_1'|\ge 3((1-\eps)|S|-1.3\eps |S|) \ge 3(1-2.3\eps)|S|$. That is,
\begin{equation}
\sum_{i\ge 1} in_i \ge (3-6.9\eps)|S|. \lab{S-Sbar}
\end{equation}
By~\eqn{SS} and noting that $2i-1\ge i$ for every $i\ge 2$, 
\[
(2-\eps)|S|\le n_1 + (2+\eps)\sum_{i\ge 2} n_i -\sum_{i\ge 2} in_i.
\]
Plugging the lower bound for $\sum_{i\ge 1} in_i $ from~\eqn{S-Sbar} yields
\[
(2+\eps)\sum_{i\ge 1}n_i\ge n_1+(2+\eps)\sum_{i\ge 2} n_i \ge (2-\eps)|S| +(3-6.9\eps)|S| = (5-7.9\eps) |S|.
\]
Thus,
\[
|N(S)|=\sum_{i\ge 1}n_i \ge \frac{5-7.9\eps}{2+\eps} |S| >2.1 |S|, 
\]
as $\eps<0.05$.
This contradicts with $|N(S)|\le |S|+x_1-1<2|S|$. 
It follows then that $|S|\ge \eps_0 n$. 
\end{proof}

Now, it is straightforward to finish the proof of Lemma~\ref{lem:Ham}.

\begin{proof}[Proof of Lemma~\ref{lem:Ham}]
We extend $P$ whenever possible, and if it is not possible, then $|V(P)|\ge \eps_0 n$ by Claim~\ref{claim:linear}. The vertices outside of $P$ are in a collection $\cal F$ of $o(n)$ paths and cycles. Let $v_t$ be an arbitrary vertex outside of $P$ that is either an end vertex of a path, or any vertex in a cycle.  If the semi-random process selects a vertex $u_t\in {\tt End}$ then, by performing Pos\'{a} rotations, we extend $P$ to a longer path by absorbing a path or a cycle in $\cal F$ that was not in $P$. The number of components in $\cal F$ goes down by 1. Otherwise, we simply ignore $u_t$ and $v_t$, and repeat until eventually $u_t\in {\tt End}$. Since $|{\tt End}|\ge \eps_0 n$, the probability that $u_t\in {\tt End}$ is at least $1/\eps_0$. Hence, it takes $O(1)$ trials on average to absorb a path or a cycle. Since there are only $o(n)$ paths or cycles to be absorbed, it follows immediately from Chernoff bound that a.a.s.\ an additional $o(n)$ edges are enough to be added to $G^*$ to make it Hamiltonian.
\end{proof}

\subsection{Preparation for the Proof of Lemma~\ref{lem:2-matching}}\label{sec:upper_2-matching_preparation}

Our aim now is to prove that $G^*$ has a 2-matching with $o(n)$ components.  In order to achieve it, we will apply the consequence of the Tutte-Berge matching formula~\cite[Theorem~30.7]{Shrijver} to $\hat G_{\tau_4}$, which is a simple graph and a subgraph of $G^*$.  However, we need one more definition before we can state it. 

Given a simple graph $G$, let $\kappa(G)$ be the number of edges in a maximum 2-matching of $G$ (that is, $\kappa(G)$ is the \textbf{size} of a maximum 2-matching). The Tutte-Berge matching theorem implies the following.

\begin{thm}\label{thm:Tutte-Berge}
Let $G$ be a simple graph on the vertex set $[n]$. Then, 
\[
\kappa(G)=\min\left\{n+ |U| - |S| +\sum_{X} \floor{\frac{e(X,S)}{2}}\right\},
\]
where $U$ and $S$ are disjoint subsets of $[n]$, $S$ is an independent set, and $X$ ranges over the components of $G-U-S$.
\end{thm}

Despite the fact that the above theorem provides the exact value for $\kappa(G)$, it is not so easy to apply it in the context of random graphs. Fortunately, if $G$ belongs to some family of graphs, then we get an easier property to check. We will first define the family, then prove a weaker but more workable statement, and finally show that a.a.s.\ $G^*$ belongs to the family.

\medskip

Let $\cyclic$ be the family of graphs on the vertex set $[n]$ which satisfy the following properties: there are at most $n/\ln n$ subsets $S \subseteq [n]$ with $|S| \le \ln n/10$ such that $S$ induces a connected subgraph with at least the same number of edges as the number of vertices; that is, $G[S]$ is connected and $|E(G[S])| \ge |S|$. 

\begin{cor}\label{cor:cyclic}
Suppose $G\in \cyclic$ and all 2-matchings of $G$ have more than $\gamma$ components for some $\gamma\ge 0$. Then, $G$ has vertex partition $[n]=S\cup T\cup R\cup U$ such that
\begin{itemize}
\item[(a)] $S$ is an independent set and $G[T]$ is a forest;
\item[(b)] $|S|\ge \max\{|U|, \gamma-11n/\ln n\}$;
\item[(c)] $e(S\cup T)+e(S\cup T, R)\le |T|+2|S|-2|U|-2\gamma + 33 n/\ln n$;
\item[(d)] $e(R,T)=0$.
\end{itemize}
\end{cor}

\begin{proof}
The proof is almost identical to that in~\cite{3-out} so we only briefly sketch the argument here. Let $F$ be a maximum 2-matching of $G$. Since $G\in \cyclic$, the number of cycles in $F$ is at most $n/\ln n+n/(\ln n/10)=11n/\ln n$. Let $c(F)$ and $e(F)$ denote the numbers of components and, respectively, edges in $F$. Then, 
$$
\gamma \le c(F)\le n-e(F)+11n/\ln n=n+11n/\ln n-\kappa(G). 
$$
Thus, $\kappa(G)\le n+11n/\ln n-\gamma$.

Let $S$ and $U$ be a pair of disjoint subsets of $[n]$ that minimize 
\begin{equation}
n+ |U| - |S| +\sum_{X\in {\mathcal C}} \floor{\frac{e(X,S)}{2}},\label{TB}
\end{equation}
where ${\mathcal C}$ is the set of components of $G-U-S$, and $S$ is an independent set. Let $T$ be the union of  components of $G-U-S$ that are trees, and $R=[n]-U-S-T$. By Theorem~\ref{thm:Tutte-Berge} and our earlier observation, we get that $\kappa(G) = n + |U| - |S| +\sum_{X\in {\mathcal C}} \floor{\frac{e(X,S)}{2}} \le n+11n/\ln n-\gamma$, and so
\begin{equation}
|U| - |S| +\sum_{X\in {\mathcal C}} \floor{\frac{e(X,S)}{2}} \le  11n/\ln n-\gamma. \lab{cond}
\end{equation}

By our construction, $[n]=S\cup T\cup R\cup U$ is a partition of the vertex set, and properties (a) and (d) hold. It remains to show that properties (b) and (c) also hold. 
It follows immediately from inequality~\eqn{cond} that 
$$
|S| \ge |U| + \sum_{X\in {\mathcal C}} \floor{\frac{e(X,S)}{2}} + \gamma - 11n/\ln n \ge \gamma - 11n/\ln n,
$$
since $|U| \ge 0$ and $\sum_{X\in {\mathcal C}} \floor{\frac{e(X,S)}{2}}\ge 0$.
On the other hand, by Theorem~\ref{thm:Tutte-Berge} and the fact that $(S,U)$ is chosen such that it minimizes~\eqn{TB}, we have $n\ge \kappa(G) \ge n+|U|-|S|$, which implies $|S|\ge |U|$. This shows that property~(b) holds.  

For part~(c), let $p$ and $q$ denote the number of components in $G[T]$ and, respectively, $G[R]$. Since $G\in\cyclic$, there are at most $n/\ln n$ components in $G[R]$ of order at most $\ln n/10$, and at most $10n/\ln n$ components in $G[R]$ of order greater than $\ln n/10$. It follows that $q\le 11n/\ln n$. Then, 
\[
\sum_{X\in {\mathcal C}} \floor{\frac{e(X,S)}{2}}\ge \frac{(e(S,T)-p)+(e(S,R)-q)}{2}\ge \frac{e(S,T)+e(S,R)-p-11n/\ln n}{2}.
\]
Hence,
\[
11n/\ln n-\gamma\ge |U| - |S| +\sum_{X\in {\mathcal C}} \floor{\frac{e(X,S)}{2}} \ge |U|-|S|+\frac{e(S,T)+e(S,R)-p-11n/\ln n}{2}.
\]
It follows that
\[
e(S,T)+e(S,R) \le 33n/\ln n-2\gamma+2|S|-2|U|+p.
\]
Now condition (c) follows since
\[
e(S\cup T)+e(S\cup T, R)=e(S,T) +e(T) +e(S,R) \le  e(S,T) +e(S,R) + |T|-p, 
\]
as $T$ induces a forest and $e(T,R)=0$.
\end{proof}

\medskip

Let us now show that $\hat G_{\tau_4}$ belongs to the family $\cyclic$ and so Corollary~\ref{cor:cyclic} can be applied.

\begin{lemma}\label{lem:gstar_cyclic}
A.a.s.\ $\hat G_{\tau_4}\in \cyclic$.
\end{lemma}
\begin{proof}
Let ${\cal Z}$ be the family of sets $S$ with $|S|\le \ln n/10$ where $S$ induces a connected subgraph of $\hat G_{\tau_4}$ with at least $|S|$ edges, and let $Z=|{\cal Z}|$. We will show that $\ex [Z]=o(n/\ln n)$ which proves the lemma as it implies that $Z \le n/\ln n$ by Markov's inequality.

For a given $S \subseteq [n]$ with $|S| \le \ln n /10$, let $X_S$ be the indicator random variable that $S$ induces a connected subgraph of $\hat G_{\tau_3}$ with at least $|S|$ edges. Let $X = \sum_{S: 3 \le |S|\le\ln n/10} X_S$. 
It follows that
\begin{eqnarray*}
\ex [X] &\le& \sum_{s=3}^{\ln n / 10} {n \choose s} s^{s-2} {s \choose 2} \left( \frac {8.15}{n} \right)^s \le \sum_{s=3}^{\ln n / 10} ( 8.15e)^s = O \left( (8.15e)^{\ln n/10} \right) \\
&=& O(n^{0.36}) = o(n/\log n).
\end{eqnarray*}
(Indeed, there are $n \choose s$ sets of cardinality $s$, $s^{s-2}$ spanning trees of $K_s$, and $s \choose 2$ choices for an additional edge. By~(\ref{eq:prob_Gstar}), the probability that selected edges are present in $G^*$ is at most $(8.15n)^s$.) 

Note that $X$ counts those sets $S\in {\cal Z}$ that already satisfy the desired property in the subgraph $\hat G_{\tau_3}$. We may assume that $\hat G_{\tau_4}$ has property {\tt E}. Hence, it is sufficient to further bound the number of sets $S\in {\cal Z}$ that contain exactly one deficit vertex $v$ and induce at least one golden edge incident with $v$. 
Let $Y_S$ be the indicator variable that $S$ is such a set. Let $Y = \sum_{S: 3 \le |S|\le\ln n/10} Y_S$ and we immediately have $Z\le X+Y$.
Hence, our next task is to upper bound $\ex [Y]$. There are $|S|$ ways to choose vertex $v$ in $S$ to be the deficit vertex. Then either $v$ is incident with a loop, or a multiple edge in $G_{\tau_2}$. We will only bound $\ex[Y_S1_{A}]$ where $A$ denotes the event that the deficit vertex in $S$ is incident with a loop in  $G_{\tau_2}$; the other case can be dealt with analogously. Let $s=|S|$. There are $s^{s-2}\binom{s}{2}$ ways to specify a set of $s$ edges that must be induced by $S$. Given a specification of such $s$ edges, there are at most $s$ ways to specify one of them to be golden. The probability for that specific edge to be golden is at most $2/n<8.15/n$ (as $v$ sends out 2 golden edges in total). There could be another edge among the $s$ edges that is golden, and the conditional probability for that is at most $2/n<8.15/n$. Moreover, the probability that $v$ is incident with a loop is $O(1/n)$. It follows now that 
\[
\ex[Y1_{A}] \le \sum_{s=3}^{\ln n/10}\binom{n}{s}s^2\cdot s^{s-2}\binom{s}{2}\left(\frac{8.15}{n}\right)^s \cdot O(1/n) = O\left( \frac {\ln^2 n}{n}\right) \cdot \sum_{s=3}^{\ln n / 10} (8.15e)^s=o(1).
\]
As we already mentioned, similar calculations show that $\ex [Y1_{B}]=o(1)$, where $B$ is the event that the deficit vertex in $S$ is incident with a parallel edge in $G_{\tau_2}$. Combining all of the above, we have
$\ex [Z] \le \ex [X] +\ex [Y1_A]+\ex[Y1_{B}]=o(n/\ln n)$.
The lemma follows by Markov's inequality.
\end{proof}

\medskip

Let us fix an arbitrarily small $\eps>0$. After combining Lemma~\ref{lem:gstar_cyclic} and Corollary~\ref{cor:cyclic}, it remains to show that a.a.s.\ there is no vertex partition $S\cup T\cup U\cup R$ of $\hat G_{\tau_4}$ satisfying properties (a)--(d) in Corollary~\ref{cor:cyclic} with some $\gamma \ge \eps n$. However, the distribution of $\hat G_{\tau_4}$ is complicated. As a result, we will work on $G_{\tau_4}$ instead and use Property {\tt E}, which implies that $\hat G_{\tau_4}$ misses at most $\ln \ln n$ edges of $G_{\tau_4}$.

It will be convenient to colour edges of $G_{\tau_3}$ in one of the four colours: blue, green, red, and yellow. Recall that $G_{\tau_3}$ is constructed during the first three phases, and $\G_{\tau_4}$ is obtained by adding up to $\ln \ln n$ golden semi-edges to $G_{\tau_3}$. During the first phase, $G_{2n}$ and the corresponding directed graph $D_{2n}$ are created; $V_0$ and $V_1$ are the sets of vertices in $D_{2n}$ of in-degree 0 and, respectively, of in-degree 1. Let us colour edges of $G_{2n}$ green if their counterparts in $D_{2n}$ are directed into one of the vertices in $V_0 \cup V_1$ which we also colour green. The remaining edges are coloured blue. During the second phase graph $G_{\tau_2}$ is created; let us colour edges added during this phase red. Finally, edges added during the third phase are coloured yellow.


Let us consider any partition $[n]=S\cup T\cup U\cup R$. For any $i\in \{S,T,U, R\}$, let $\alpha_i$ be the fraction of vertices that belong to set $i$ (that is, $\alpha_i=|i|/n$) and let $\gamma_i$ be the fraction of vertices of $i$ that are green (that is, $\gamma_i = |G_i|/\alpha_i n$ where $G_i$ is the set of green vertices in set $i$). Moreover, let $\beta_i$ be the fraction of vertices in $G_i$ that received no incoming edge in $D_{2n}$ (that is, $\beta_i = |G_i \cap V_0|/|G_i|$). In order to simplify the notation, we define the following vectors: ${\bm \alpha}=(\alpha_i)_{i\in \{S,T,U,R\}}$, ${\bm \beta}=(\beta_i)_{i\in \{S,T,U,R\}}$, ${\bm \gamma}=(\gamma_i)_{i\in \{S,T,U,R\}}$. It follows immediately from the above definitions that the following properties hold:
\begin{equation}
\sum_{i\in \{S,T,U,R\}}\alpha_i=1, \quad, 0\le \gamma_i\le 1, \quad 0\le \beta_i\le 1 \quad\mbox{for all $i\in \{S,T,U,R\}$}. \lab{alpha}
\end{equation}
Next, for $i,j\in \{S,T,U,R\}$, let $b_{ij} \cdot (2\alpha_i n)$, $g_{ij} \cdot (2\alpha_i n)$, and $r_{ij} \cdot (2\beta_i+(1-\beta_i))\gamma_i\alpha_i n$ denote the numbers of blue, green and, respectively, red edges from set $i$ to set $j$. Vectors ${\bm b}=(b_{ij})_{i,j\in \{S,T,U,R\}}$, ${\bm g}=(g_{ij})_{i,j\in \{S,T,U,R\}}$, and ${\bm r}=(r_{ij})_{i,j\in \{S,T,U,R\}}$ describe the distribution of edges of a given colour between parts. Let $y_1\cdot 0.07n$ denote the number of yellow edges that are either incident to a vertex in $U$, or are induced by $R$. Let $y_2\cdot 0.07n$ denote the number of yellow edges that are induced by $T$. Note that there are $(1-y_1-y_2) \cdot 0.07n$ edges between $S$ and $R \cup T$. Hence, the vector $(y_1, y_2, 1-y_1-y_2)$ describes the distribution of yellow edges.

Let us fix ${\bm u}=({\bm \alpha}, {\bm \beta}, {\bm \gamma}, {\bm b}, {\bm g}, {\bm r}, y_1,y_2)$. Our goal is to upper bound the probability $P(\bm u)$ that there exists a partition $[n]=S\cup T\cup U\cup R$ with $|i|=\alpha_i n$ for $i\in \{S,T,U,R\}$, and subsets $i'\subseteq i$ for $i\in \{S,T,U,R\}$ with $|i'|=\gamma_i\alpha_i n$ such that the following properties hold:
\begin{itemize}
\item there are exactly $b_{ij} \cdot (2\alpha_i n)$ blue directed edges from set $i$ to set $j$;
\item there are exactly  $g_{ij} \cdot (2\alpha_i n)$ green directed edges from set $i$ to set $j$;
\item there are  exactly  $r_{ij} \cdot (2\beta_i+(1-\beta_i))\gamma_i\alpha_i n$ red directed edges from set $i$ to set $j$; 
\item there are exactly $y_1\cdot 0.07n$  yellow edges that are either incident to a vertex in $U$, or are induced by $R$;
\item there are exactly $y_2\cdot 0.07n$  yellow edges that are induced by $T$;
\item there are no yellow edges inside $S$, or between $R$ and $T$;
\item all vertices in $i'$ received at most 1 incoming green edge;
\item all vertices in $i\setminus i'$ received at least 2 incoming blue edges.
\end{itemize} 
We will show that $P({\bm u}) \le poly(n) \exp(f(\bm u)n)$ for some explicit function $f(\bm u)$. Unfortunately, this function is quite involved so we will define it in the next section. 

\subsection{Property $\P$}\label{sec:function}

Let $\eps_0=2^{-32}$. Let us start with an observation that, due to Lemma~\ref{lem:aas}, we may assume that the parameter ${\bm u}$ is of a specific form, that is, it satisfies the following constraints:
\begin{align}
&-\eps_0<\sum_{i\in \{S,T,U,R\}}\gamma_i\alpha_i - 3e^{-2}<\eps_0\label{eq1}\\
&-\eps_0<\sum_{i\in \{S,T,U,R\}}\beta_i\gamma_i\alpha_i - e^{-2}<\eps_0\label{eq2}\\
&-\eps_0<\sum_{i\in \{S,T,U,R\}} 2\alpha_i \sum_{j\in \{S,T,U,R\}} g_{ij} - 2e^{-2}<\eps_0,\label{eq3}
\end{align}  
Indeed, equations~(\ref{eq1}) and~(\ref{eq2}) follow from the fact that a.a.s.\ $|V_0|+|V_1| = (3e^{-2}+o(1)) n$ and, respectively, $|V_0| = (e^{-2}+o(1)) n$ (Lemma~\ref{lem:aas}(a)); equation~(\ref{eq3}) follows from the fact that the number of green edges is equal to $|V_1|$ and so a.a.s.\ it is asymptotic to $2e^{-2}n$.
We also have the following set of obvious constraints:
\begin{align}
&\sum_{j\in \{S,T,U,R\}} (b_{ij}+g_{ij}) =1 , \quad \mbox{for all $i\in \{S,T,U,R\}$} \lab{sum1} \\
&\sum_{j\in \{S,T,U,R\}} r_{ij} =1 , \quad \mbox{for all $i\in \{S,T,U,R\}$} \\
&y_1+y_2\le 1\\
&{\bm \alpha}, {\bm \beta}, {\bm \gamma}, {\bm b}, {\bm g}, {\bm r}, y_1, y_2 \in [0,1]. 
\end{align}
(For the ease of notation, we write that a vector is in $[0,1]$ when every component of the vector is in $[0,1]$.)
As we only consider partitions satisfying properties (a)--(d) stated in Corollary~\ref{cor:cyclic}, we additionally require that
\begin{align}
&\alpha_S\ge \alpha_U, \\
&2\alpha_Tb_{TT}+2\alpha_Tg_{TT}+\gamma_T\alpha_T(2\beta_T+(1-\beta_T))r_{TT} + 0.07y_2 \le \alpha_T+\min\{\eps_0, \alpha_T\}, \\
&c_{SS}=c_{RT}=c_{TR}=0,\quad \mbox{for all $c\in \{b,g,r\}$}.
\end{align}
The first constraint comes immediately from property~(b),  and the last constraint follows from properties~(a) and~(d). The second constraint comes from the fact that $e(T)\le |T|$ in $\hat G_{\tau_4}$ required by property~(a), which together with Property {\tt E} imply that $e(G)\le |T|+\eps_0 n$ and $e(G)\le 2|T|$ in $G_{\tau_4}$ (note there can be at most $|T|$ loops or double edges induced by $T$). Finally, let us note that Properties~(c) and {\tt E}, and Lemma~\ref{lem:number_of_edges} imply that a.a.s.\ the number of edges incident with $U$ or induced by $R$ is at least
\begin{align*}
|E(G_{\tau_4})| - & e(S\cup T) - e(S \cup T,R) \\
& \ge (2+4e^{-2}+0.07+o(1))n -|T|-2|S|+2|U| + 2\gamma-33n/\ln n -\ln \ln n\\
& \ge (2+4e^{-2}+0.07)n-|T|-2|S|+2|U|\\
& =(4e^{-2}+0.07+4\alpha_U+\alpha_T+2\alpha_R)n.
\end{align*}
This yields the following constraint:
\begin{align}
2\alpha_U +\gamma_U\alpha_U(1+\beta_U) + 2\alpha_S(b_{SU}+g_{SU}) +\gamma_S\alpha_S r_{SU}(1+\beta_S) + 2\alpha_T(b_{TU}+g_{TU}) & \nonumber\\
            + \gamma_T\alpha_Tr_{TU}(1+\beta_T)+  2\alpha_R(b_{RU}+b_{RR}+g_{RU}+g_{RR}) & \nonumber\\
+ \gamma_R\alpha_R(r_{RU}+r_{RR})(1+\beta_R) +0.07 y_1&\nonumber\\
\ge 4e^{-2} +0.07 + 4\alpha_U + \alpha_T + 2\alpha_R. & \lab{constraint1}
\end{align} 
For $i\in\{S,T,U,R\}$, the number of blue edges coming into set $i$ must be at least $2\alpha_i(1-\gamma_i) n$, as every vertex in $i\setminus (V_0\cup V_1)$ must receive at least 2 blue edges. This yield the following set of constraints:
\begin{equation}
    \sum_{j\in\{S,T,U,R\}} 2\alpha_j  b_{ji} \ge 2\alpha_i (1-\gamma_i),\quad \mbox{for all $i\in\{S,T,U,R\}$.}\label{blue}
\end{equation}
Finally,
the number of green edges coming into each set satisfies the following constraints:
\begin{equation}
\sum_{j\in\{S,T,U,R\}} 2\alpha_j  g_{ji} = \alpha_i \gamma_i(1-\beta_i),\quad \mbox{for all $i\in\{S,T,U,R\}$.}\lab{constraint}
\end{equation}

\medskip

Now, we are ready to show that $P({\bm u}) \le poly(n) \exp(f(\bm u)n)$ for some explicit function $f(\bm u)$. 
Given a vector of non-negative real numbers ${\bm a}$ with $\sum_i a_i = 1$, let $H({\bm a})=-\sum_{i} a_i\ln a_i$. If ${\bm a}=(a_1,a_2)$, then we simply write $H(a_1)$ for $H({\bm a})$.
By convention, we set $0\log(0)=0$, for any $a >0$ we set $a\log(0)=-\infty$, and we treat $-\infty< x$ for every real number $x$.

Given ${\bm u}$, there are $\binom{n}{\alpha_S n, \alpha_T n, \alpha_U n, \alpha_R n} =poly(n) \exp(H(\alpha_S,\alpha_T,\alpha_U,\alpha_R)n)$ choices for sets $S$, $T$, $U$, and $R$. Given $S$, $T$, $U$, $R$, there are 
\[
\prod_{i\in\{S,T,U,R\}}\binom{\alpha_i n}{\gamma_i \alpha_i n} =poly(n)\prod_{i\in\{S,T,U,R\}} \exp(H(\gamma_i) \alpha_i n) =poly(n) \exp \left( n \sum_{i\in\{S,T,U,R\}} H(\gamma_i) \alpha_i \right)
\] 
ways to choose $G_S$, $G_T$, $G_U$ and $G_R$. 
The probability that the number of blue and green edges going out of $S$ into each part of $S$, $T$, $U$, $R$ is precisely as prescribed by $\bm u$ is equal to $poly(n) \exp(f_S n)$, where
\begin{align*}
f_S=&2\alpha_S \Big( H(b_{SU},b_{ST},b_{SR},g_{SU},g_{ST},g_{SR}) + b_{SU}\ln((1-\gamma_U)\alpha_U) +b_{ST}\ln ((1-\gamma_T)\alpha_T)\\
&+b_{SR}\ln ((1-\gamma_R)\alpha_R) +g_{SU}\ln(\gamma_U\alpha_U)+g_{ST}\ln(\gamma_T\alpha_T) +g_{SR}\ln(\gamma_R\alpha_R)   \Big).
\end{align*}
Indeed, there are $2 \alpha_S n$ edges going out of $S$ that are blue or green. We need to partition them into 6 classes depending on their colour and to which part they go to. This gives us the term $2\alpha_S H(b_{SU},b_{ST},b_{SR},g_{SU},g_{ST},g_{SR})$. For each $i\in\{T,U,R\}$, there are $2\alpha_S b_{Si} n$ blue edges that need to go to blue vertices of $i$ (hence terms $2\alpha_S b_{Si} \ln((1-\gamma_i)\alpha_i)$) and there are $2\alpha_S g_{Si} n$ green edges that need to go to green vertices of $i$ (hence terms $2\alpha_S g_{Si} \ln(\gamma_i \alpha_i)$).
Similarly, the probabilities that the number of blue and green edges going out of $T$, $U$, $R$ into other parts is precisely as encoded by $\bm u$ are $poly(n) \exp(f_T n)$, $poly(n) \exp(f_Un)$ and, respectively, $poly(n) \exp(f_Rn)$, where
\begin{align*}
f_T=&2\alpha_T \Big( H(b_{TS},b_{TT},b_{TU},g_{TS},g_{TT},g_{TU}) + b_{TS}\ln((1-\gamma_S)\alpha_S) +b_{TT}\ln ((1-\gamma_T)\alpha_T)\\
&+b_{TU}\ln ((1-\gamma_U)\alpha_U) +g_{TS}\ln(\gamma_S\alpha_S)+g_{TT}\ln(\gamma_T\alpha_T) +g_{TU}\ln(\gamma_U\alpha_U)   \Big),
\end{align*}
\begin{align*}
f_U=&2\alpha_U \Big( H(b_{US},b_{UT},b_{UU},b_{UR},g_{US},g_{UT}, g_{UU}, g_{UR}) + b_{US}\ln((1-\gamma_S)\alpha_S) \\
&+b_{UT}\ln ((1-\gamma_T)\alpha_T) +b_{UU}\ln ((1-\gamma_U)\alpha_U)+b_{UR}\ln ((1-\gamma_R)\alpha_R) +g_{US}\ln(\gamma_S\alpha_S) \\
&+g_{UT}\ln(\gamma_T\alpha_T) +g_{UU}\ln(\gamma_U\alpha_U)+g_{UR}\ln(\gamma_R\alpha_R)   \Big),
\end{align*}
and
\begin{align*}
f_R=&2\alpha_R \Big( H(b_{RS},b_{RU},b_{RR},g_{RS},g_{RU},g_{RR}) + b_{RS}\ln((1-\gamma_S)\alpha_S) +b_{RU}\ln ((1-\gamma_U)\alpha_U)\\
&+b_{RR}\ln ((1-\gamma_R)\alpha_R) +g_{RS}\ln(\gamma_S\alpha_S)+g_{RU}\ln(\gamma_U\alpha_U) +g_{RR}\ln(\gamma_R\alpha_R)   \Big).
\end{align*}
Given that, the probabilities that the number of red edges going out of $S$, $T$, $U$, $R$ into each part of $S$, $T$, $U$, $R$ exactly as dictated by $\bm u$ are $poly(n) \exp(g_S n)$, $poly(n) \exp(g_T n)$, $poly(n) \exp(g_Un)$ and, respectively, $poly(n) \exp(g_Rn)$, where
\begin{align*}
g_S=&\alpha_S\gamma_S(2\beta_S+(1-\beta_S))\Big(H(r_{SU},r_{ST},r_{SR})+r_{SU}\ln \alpha_U+r_{ST}\ln \alpha_T+r_{SR}\ln \alpha_R\Big), \\
g_T=&\alpha_T\gamma_T(2\beta_T+(1-\beta_T))\Big(H(r_{TS},r_{TT},r_{TU})+r_{TS}\ln \alpha_S+r_{TT}\ln \alpha_T+r_{TU}\ln \alpha_U\Big), \\
g_U=&\alpha_U\gamma_U(2\beta_U+(1-\beta_U))\Big(H(r_{US},r_{UT},r_{UU}, r_{UR})+r_{US}\ln \alpha_S+r_{UT}\ln \alpha_T \\
&\qquad\qquad\qquad\qquad\qquad\qquad +r_{UU}\ln \alpha_U+r_{UR}\ln \alpha_R\Big), \\
g_R=&\alpha_R\gamma_R(2\beta_R+(1-\beta_R))\Big(H(r_{RS},r_{RU},r_{RR})+r_{RS}\ln \alpha_S+r_{RU}\ln \alpha_U+r_{RR}\ln \alpha_R\Big).
\end{align*}

\medskip

In order to continue our computations, we need the following auxiliary lemma on the ``balls into bins'' model.

\begin{lemma}\lab{lem:balls_in_bins}
Fix $\alpha > 0$ and suppose that $\alpha n$ balls are thrown independently and uniformly at random into $n$ bins. 
\begin{enumerate}
\item[(a)] If $\alpha>2$, then the probability that every bin receives at least two balls is asymptotic to $poly(n)\exp(t(\alpha)n)$
with $t(\alpha)=\lambda-\alpha+\alpha\ln(\alpha/\lambda)+\ln(1-e^{-\lambda}-\lambda e^{-\lambda})$,
where $\lambda = \lambda(\alpha) > 0$ is the unique solution of the following equation:
\[
\frac{\lambda(1-e^{-\lambda})}{1-e^{-\lambda}-\lambda e^{-\lambda}}=\alpha.
\]
\item[(b)] If $\alpha\le 1$, then the probability that every bin receives at most one ball is asymptotic to $poly(n)\exp(\kappa(\alpha)n)$,
where
$
\kappa(\alpha)=-\alpha-(1-\alpha)\ln (1-\alpha).
$
\item[(c)] If $\alpha=2$, the probability that every bin receives exactly two balls is asymptotic to $poly(n)\exp((\ln 2-2)n)$.
\end{enumerate}
\end{lemma}
Before we prove the lemma, let us note that 
$$
f(\lambda) := \frac{\lambda(1-e^{-\lambda})}{1-e^{-\lambda}-\lambda e^{-\lambda}} = \frac {\lambda (1-(1-\lambda+O(\lambda^2)))}{1 - (1-\lambda+\lambda^2/2+O(\lambda^3))(1+\lambda)} = \frac {\lambda^2 + O(\lambda^3)}{\lambda^2/2+O(\lambda^3)} = 2 +O(\lambda),
$$
so $\lim_{\lambda \to 0^+} f(\lambda) = 2$. It is also straightforward to see that $\lim_{\lambda \to \infty} f(\lambda) = \infty$ and $f(\lambda)$ is an increasing function of $\lambda$. Hence, indeed, $\lambda=\lambda(\alpha)$ is well defined.
For convenience, we define $\lambda(2)=0$ and set
\[
 t(2)=\lim_{\alpha\to 2+} t(\alpha)=\ln 2-2.
\]
This definition of $t:[2,\infty)\to {\mathbb R}$ unifies parts (a) and (c) in the lemma above.

\begin{proof}
Suppose that $\alpha\ge 2$. Let $K$ be the truncated Poisson variable with parameter $\lambda = \lambda(\alpha)$ and truncated at 2, that is,
\[
\pr(K=j)=\frac{e^{-\lambda}\lambda^j}{j!(1-e^{-\lambda}-\lambda e^{-\lambda})}, \qquad\mbox{for every integer $j\ge 2$.}
\]
It follows that
\begin{eqnarray*}
\ex K &=& \sum_{j\ge 2} j \cdot \pr(K=j) = \sum_{j\ge 2} \frac{e^{-\lambda}\lambda^j}{(j-1)!(1-e^{-\lambda}-\lambda e^{-\lambda})} \\
&=& \frac{\lambda}{1-e^{-\lambda}-\lambda e^{-\lambda}} \sum_{j\ge 1} \frac{e^{-\lambda}\lambda^j}{j!} = \frac{\lambda(1-e^{-\lambda})}{1-e^{-\lambda}-\lambda e^{-\lambda}} = \alpha,
\end{eqnarray*}
by the definition of $\lambda$.

Let $k_1,\ldots, k_n$ be $n$ independent copies of $K$. Then, by Gnedenko's local limit theorem~\cite{lll},
\begin{eqnarray*}
\Theta(n^{-1/2}) &=& \pr \left( \sum_{i=1}^n k_i=\alpha n \right) = \sum_{\substack{j_1\ge 2, \ldots, j_n\ge 2\\ \sum_{i=1}^n j_i =\alpha n}} \prod_{i=1}^n \frac{e^{-\lambda} \lambda^{j_i}}{j_i!(1-e^{-\lambda}-\lambda e^{-\lambda})} \\
&=& \frac{e^{-\lambda n}\lambda^{\alpha n}}{(1-e^{-\lambda}-\lambda e^{-\lambda})^n}{\sum}^*,
\end{eqnarray*}
where
\[
{\sum}^*=  \sum_{\substack{j_1\ge 2, \ldots, j_n\ge 2\\ \sum_{i=1}^n j_i =\alpha n}} \prod_{i=1}^n \frac{1}{j_i!}. 
\]
Hence,
\[
{\sum}^* = poly(n)  \frac{(1-e^{-\lambda}-\lambda e^{-\lambda})^n}{e^{-\lambda n}\lambda^{\alpha n}}.
\]

Consider now throwing $\alpha n$ balls independently and uniformly at random into $n$ bins. By Stirling's formula ($x! = poly(x) (x/e)^x$), the probability that every bin receives at least 2 balls is equal to
\begin{eqnarray*}
\sum_{\substack{j_1\ge 2, \ldots, j_n\ge 2\\ \sum_{i=1}^n j_i=\alpha n}} \binom{\alpha n}{j_1,\ldots, j_n} \ n^{-\alpha n} 
&=& \frac{(\alpha n)!}{n^{\alpha n}} {\sum}^* = poly(n) e^{-\alpha n} \alpha^{\alpha n} {\sum}^* \\
&=& poly(n) e^{-\alpha n} \alpha^{\alpha n} \frac{(1-e^{-\lambda}-\lambda e^{-\lambda})^n}{e^{-\lambda n}\lambda^{\alpha n}}=poly(n)\exp(t(\alpha) n). 
\end{eqnarray*}
This completes the proof of part~(a).

To show part~(b), suppose that $\alpha\le 1$. The probability that every bin receives at most one ball is equal to
\[
\frac{(n)_{\alpha n}}{n^{\alpha n}} = \frac {n!}{(n-\alpha n)! n^{\alpha n}} =poly(n)\exp(\kappa(\alpha)n),
\]
where $(x)_j=\prod_{i=0}^{j-1}(x-j)$ denotes the $j$-th falling factorial. 

To show part~(c), note that the probability that every bin receives exactly two balls is equal to
\[
\frac{(2n)!/2^n}{n^{2n}} = poly(n)\exp((\ln2-2)n).
\]
This finishes the proof of the lemma.
\end{proof}

\medskip

We are now back to our problem. With Lemma~\ref{lem:balls_in_bins} at hand, we will be able to prove the following claims.

\begin{claim}\lab{claim:blue}
The probability that all vertices in $[n]\setminus(G_S\cup G_T\cup G_U\cup G_R)$ receive at least two blue edges is equal to $poly(n) \exp \left( n \sum_{i\in \{S,T,U,R\}} w_i \right) $, where
\[
w_i=(1-\gamma_i)\alpha_i\Big(\lambda_i-d_i+d_i\ln(d_i/\lambda_i)+\ln(1-e^{-\lambda_i}-\lambda_ie^{-\lambda_i})\Big),
\]
\[
d_i=\frac{\sum_{j\in\{S,T,U,R\}} 2\alpha_jb_{ji}}{(1-\gamma_i)\alpha_i},
\]
and $\lambda_i=\lambda_i(d_i) > 0$ is the unique solution of the following equation:
\[
\frac{\lambda_i(1-e^{-\lambda_i})}{1-e^{-\lambda_i}-\lambda_ie^{-\lambda_i}}=d_i.
\]
\end{claim}

Before we move to the proof, let us remark that, by~\eqn{blue}, for every $i\in\{S,T,U,R\}$ we have $d_i\ge 2$ and so $\lambda_i$ is well defined.

\begin{proof}
Note that for each $i\in\{S,T,U,R\}$, the number of blue edges coming into $i\setminus G_i$ is equal to
$\sum_{j\in \{S,T,U,R\}} 2\alpha_j n b_{ji}$. Moreover, $|i\setminus (V_0\cup V_1)|=(1-\gamma_i)\alpha_i n$. The claim follows immediately from Lemma~\ref{lem:balls_in_bins}(a) applied with $\alpha=\sum_{j\in \{S,T,U,R\}} 2\alpha_j  b_{ji}/(1-\gamma_i)\alpha_i = d_i$ and the number of balls equal to $(1-\gamma_i)\alpha_i n$.
\end{proof}

\medskip

\begin{claim}\lab{claim:green}
The probability that all vertices in $G_S\cup G_T\cup G_U\cup G_R$ receive at most one green edge is equal to $poly(n) \exp \left( n \sum_{i\in \{S,T,U,R\}} \tilde w_i \right) $, where
\[
\tilde w_i= \alpha_i\gamma_i \Big(-1+\beta_i-\beta_i\ln \beta_i\Big).
\]
\end{claim}

\begin{proof}
Note that for each $i\in\{S,T,U,R\}$, the number of green edges coming into $i\cap(V_0\cup V_1)$ is
$(1-\beta_i)\gamma_i\alpha_i n$. Moreover, $|i\cap(V_0\cup V_1)|=\gamma_i \alpha_i n$. The claim follows immediately from Lemma~\ref{lem:balls_in_bins}(b) applied with $\alpha=(1-\beta_i)\gamma_i\alpha_i/\gamma_i\alpha_i=1-\beta_i$ and the number of bins equal to $\gamma_i\alpha_i n$.
\end{proof}

\medskip

\begin{claim}\lab{claim:yellow}
The probability that there are exactly $0.07y_1n$ yellow edges incident with $U$ or induced by $R$, and exactly $0.07y_2n$ yellow edges induced by $T$ is equal to $poly(n) \exp(h n)$, where
\begin{align*}
h=0.07\ln(0.07) &+
0.07y_1\ln\left(\frac{\alpha_U^2+2\alpha_U(1-\alpha_U)+\alpha_R^2}{0.07 y_1}\right)+
0.07y_2\ln\left(\frac{\alpha_T^2}{0.07y_2}\right)\\
&+ 0.07(1-y_1-y_2)\ln\left(\frac{2\alpha_S(\alpha_T+\alpha_R)}{0.07(1-y_1-y_2)}\right).
\end{align*}
\end{claim}

\begin{proof}
Recall that there are $0.07n$ yellow edges in total, so the remaining $0.07(1-y_1-y_2)n$ yellow edges are between $S$ and $R \cup T$. The probability in the claim is equal to
\[
\binom{\binom{\alpha_Un}{2}+\alpha_U(1-\alpha_U)n^2+\binom{\alpha_Rn}{2}}{0.07y_1n}\binom{\binom{\alpha_Tn}{2}}{0.07y_2n}\binom{\alpha_Sn(\alpha_T+\alpha_R)n}{0.07(1-y_1-y_2)n} \binom{\binom{n}{2}}{0.07n}^{-1},
\]
which is equal to $poly(n) \exp(hn)$ by Stirling's formula.
\end{proof}

\medskip

Combining everything together, it follows that $P({\bm u})=poly(n)\exp(f(\bm u)n)$, where
\[
f(\bm u)=H(\alpha_S,\alpha_T,\alpha_U,\alpha_R)+\sum_{i\in\{S,T,U,R\}}\big(\alpha_iH(\gamma_i)+f_i+g_i+w_i+\tilde w_i\big) +h.
\]

Finally, we are ready to state property $\P$. 

\begin{definition}[Property~$\P$]
Suppose there exists $\delta>0$ such that $f({\bm u})<-\delta$ for all vectors ${\bm u}$ subject to~\eqn{alpha}--\eqn{constraint} and $\alpha_R\le 0.995$. 
\end{definition}
Let us remark that the reason to separate $\alpha_R$ from 1 in the definition of Property ${\cal P}$ is that the probability of a specified vertex partition $S\cup T\cup U\cup R$ satisfying Corollary~\ref{cor:cyclic}(a)--(d) will not be exponentially small when $S$, $T$, and $U$ are all of sub-linear size, and thus $f$ is not bounded away from 0 in the entire region~\eqn{alpha}--\eqn{constraint}. 

\subsection{Proof of Lemma~\ref{lem:2-matching}}\label{sec:upper_2matching}

\begin{proof}[Proof of Lemma~\ref{lem:2-matching}]
Suppose that property $\P$ holds, that is, there exists $\delta>0$ such that $f({\bm u})<-\delta$ for all vectors ${\bm u}$ subject to~\eqn{alpha}--\eqn{constraint} and $\alpha_R\le 0.995$. Our goal is to show that a.a.s.\ $\hat G_{\tau_4}$ has a 2-matching with $o(n)$ components. Fix $\eps>0$. As mentioned earlier, after combining Lemma~\ref{lem:gstar_cyclic} and Corollary~\ref{cor:cyclic}, it remains to show that a.a.s.\ there is no vertex partition $S\cup T\cup U\cup R$ of $\hat G_{\tau_4}$ satisfying properties (a)--(d) in Corollary~\ref{cor:cyclic} with some $\gamma \ge \eps n$ and $|R|\le 0.995n$. 

The expected number of partitions $S\cup T\cup U\cup R$ satisfying (a)--(d) with $\gamma \ge \eps n$ and $|R|\le 0.995n$ is at most
\begin{equation}
\sum_{{\bm u}} P({\bm u})= \sum_{{\bm u}} poly(n)\exp(f(\bm u)n), \lab{expec}
\end{equation}
where the sum is over all possible values of ${\bm u}$ satisfying constraints~\eqn{alpha}--\eqn{constraint} and $\alpha_R\le 0.995$,  we have
 $f({\bm u})<-\delta/2$ for all ${\bm u}$ in the range of summation of~\eqn{expec} restricted to $\alpha_R\le 0.995$. The number of possible values of ${\bm u}$ in the summation is clearly $poly(n)$. Hence, the expected number of partitions $S\cup T\cup U\cup R$ satisfying (a)--(d) where $|R|\le 0.995 n$ is 
$$
 \sum_{\bm u} poly(n)\exp(f(\bm u) n) = poly(n)\exp( -\delta n / 2) = o(1). 
$$
 
It only remains to consider partitions $S\cup T\cup U\cup R$ satisfying (a)--(d) with $|R|>0.995n$. Let
\begin{align*}
  x_1 &\quad \mbox{denote the number of edges between $S$ and $T$};\\
  x_2 &\quad \mbox{denote the number of edges between $U$ and $T$};\\
  x_3 &\quad \mbox{denote the number of edges between $S$ and $U$};\\
  x_4 &\quad \mbox{denote the number of edges between $S$ and $R$}.
\end{align*}
 Since the minimum degree of $\hat G_{\tau_4}$ is at least 4, $S$ induces an independent set, and $T$ induces a forest, we get that
 \[
 x_1+x_2+2e(T)\ge 4|T|,\quad  e(T)<|T|,\quad \text{ and } \quad x_1+x_3+x_4\ge 4|S|.
 \]
By property~(c) and the fact that $\gamma \ge \eps n$,  we get that $x_1+e(T)+x_4 \le |T|+2|S|-2|U|$. Hence,
\begin{align}
2|T|+2|S|-2|U| &+x_1+x_2+x_3 > |T|+2|S|-2|U| +x_1+x_2+x_3+e(T)\nonumber\\
& \ge 2x_1+x_2+x_3+x_4+2e(T)\ge 4(|S|+|T|).
\end{align}
It follows that $x_1+x_2+x_3\ge 2(|S|+|T|+|U|) = 2(|S\cup T\cup U|)$, that is, $S\cup T\cup U$ induces at least $2|S\cup T\cup U|$ edges. However, by 
Lemma~\ref{lem:aas}(c), this does \emph{not} happen a.a.s.\ for any partition with $|S\cup T\cup U| \le 0.005n$.
\end{proof}

\subsection{Numerical support}\label{sec:numerical}

The goal of this section is to provide a numerical evidence that property $\mathcal{P}$ holds. The optimization problem was carefully investigated using the code written in the Julia language~\cite{Julia}, \texttt{JuMP.jl} package~\cite{Jump} with \texttt{Ipopt} solver~\cite{Ipopt}. The optimization problem we needed to face is challenging for the following reasons.

Firstly of all, it involves a non-convex optimization problem which potentially has many local optima (we numerically confirmed that this is the case in our problem). In order to overcome this challenge, we used a standard multi-start~\cite{Multistart} approach for solving global optimization problems. However, due to a stochastic nature of the heuristic search procedure used in this process, it means that the results we obtained are only heuristic in nature. In other words, the numerical results we obtained strongly suggest that the desired property holds but this is, unfortunately, not a formal proof of this.

Second of all, the objective function contains terms of the form $x\ln(x)$ which have derivatives tending to $\infty$ as $x\to0$. This creates a challenge when solving the problem using numerical methods. More importantly, in the problem  there are some local optima for which some variables are equal to zero. In order to overcome this problem, we relaxed the original problem by replacing $x\ln(x)$ with some other function $f(x) \le x \ln(x)$ (we need this property as we deal with a maximization problem and terms of the form $x\ln(x)$ appear with a negative sign in the objective function).  Function $f(x)$ should be a quadratic function near $0$, its value and the values of its first and second derivatives should match in the point of change of the formula. The exact function we ended up using as a relaxation of $x\ln(x)$ is:
\begin{equation*}
f(x) = \begin{cases}
2^{31}x^2+\ln(2^{-32})x-2^{-33} &\text{if $0 \le x < 2^{-32}$}\\
x\ln(x) &\text{if $x \geq 2^{-32}$}
\end{cases}.
\end{equation*}

The third challenge is that the optimization problem for most of the variables allows the domain to be $[0,1]$ and we have $\ln(x)$ occurring in multiple places of the formulation of the objective function (and also other than $x\ln(x)$ which is handled by the relaxation described above). This poses another challenge when the solver performs a local search in the points near the boundary of the admissible set. In such cases a logarithm of negative value might be considered (note that the solver evaluates the objective function for points contained in some small neighbourhood of a current potential solution before ensuring that the constraints are satisfied; as a result, if points close to $0$ are considered, such neighbourhood could contain negative values), which leads to errors when performing the computation. In order to overcome this problem, we apply the transformation given by the formula
$$
g(x) = \frac {1}{2} \left( \sin \left( \pi \left(x- \frac 12\right) \right)+1 \right)
$$
to every variable that is constrained to the interval $[0,1]$, before passing it for the evaluation of the objective function and constraints. Note that this transformation is a bijection from the interval $[0,1]$ into the interval $[0,1]$ but it guarantees that if some decision variable is tested outside the $[0,1]$ interval it is transformed back to $[0,1]$ interval (such values are rejected later anyway due to the constraints but are tested during the optimization process which causes no error). Also note that the transformation we use is an analytic function, which means that it does not introduce additional problems when calculating the first or the second derivatives of the objective functions or constraints.

In order to explore the solution space thoroughly, we have performed two optimization processes. In the first one, we tested the interior of the solution space, that is, all decision variables that are restricted to $[0,1]$ were in fact constrained even further to be in the $[0.005, 0.995]$ interval. In the second optimization scenario, we did not impose these additional constraints and all the variables were allowed to be taken from their original domain.
The largest local optimum found across both scenarios was $-0.000722123670503$ (we report the value of the original objective function, before the relaxation). It was clearly separated from the boundary; indeed, all decision variables restricted to the interval $[0,1]$ actually lied in the $[0.0032,0.9586]$ interval. This is consistent with a theoretical understanding of the problem; it is expected that there is no problem with the boundary. In both scenarios there were some additional local optima (two in the first scenario and four in the second) but all of them were smaller than the one we report above.

In order to make sure that our results are stable we tested several different values for $\epsilon_0$, various relaxation functions $f$ and space transformation functions $g$, and many separation margins from the boundary. In all cases we consistently obtained that the best local optimum found was below zero. Therefore, it provides a strong numerical support for the conjecture that the objective value of our optimization problem is negative, that is, property $\mathcal{P}$ holds.

We independently tested if the third phase (where $0.07n$ semi-random edges are sprinkled) is required for $f({\bm u})$ to be negative and bounded away from zero. Denote by $\hat{\bm u}$ the best solution for our original problem we found; it satisfies $f(\hat{\bm u})<0$. However, if $0.06n$ edges are added instead of $0.07n$, then the best solution that solver is able to find is a point ${\bm u}^*$ with $f({\bm u}^*)>0$. This time, all $[0,1]$-constrained variables in ${\bm u}^*$ turned out to be in the interval $[0.0358, 0.8691]$.
We also checked the relationship between points $\hat{\bm u}$ and ${\bm u}^*$. Both points are very close to each other ($\|{\bm u}^*-\hat{\bm u}\|_{\infty}=0.0024$), which means that the results are stable. Having said that, they are clearly not identical as changing the number of random edges added during the third phase affects the constraints of our optimization problem. In particular, point ${\bm u}^*$ is not feasible for the process involving adding $0.07n$ random edges.

\section{Lower bound}\label{sec:lower_bound}

As it was done in the argument for an upper bound, it will also be convenient to work with the directed graph $D_t$ underlying $G_t$. For each edge $u_tv_t$ that is added to $G_t$ at time $t$, we put a directed edge from $v_t$ to $u_t$ in $D_t$ (recall that $u_t$ is a random vertex selected by the semi-random graph process and $v_t$ is a vertex selected by the player). The existing lower bound for $\time$ that was observed in~\cite{process1} follows from the fact that in order to construct a Hamilton cycle, the player has to create a graph with minimum degree at least 2. However, this trivial necessary condition alone requires $(\ln 2+\ln(1+\ln2) + o(1)) n$ steps. Indeed, in order to reach a graph with minimum degree 2, the player has to play greedily during the first part of the game by selecting vertices of $G_t$ that are of degree 0. This part of the game ends at step $(\ln 2 + o(1))n$ a.a.s. From that point on, she continues playing greedily by selecting vertices of degree 1 which requires additional $(\ln(1+\ln2) + o(1)) n$ steps a.a.s. 

In order to improve the lower bound (unfortunately, only by a hair) we will use another trivial observation. We will call a vertex $x$ in $D_t$ \textbf{problematic} if it is of in-degree at least 3 (out-degree of $x$ is not important) with the in-neighbours $y_1, y_2, y_3$ (if $x$ has in-degree larger than 3, then these are the \emph{first} three in-neighbours sorted by the time when they were added to the graph), each of them of out-degree 1 and in-degree 1. Since $y_i$'s are of degree 2 in the underlying graph $G_t$, the three edges $y_i x$ must be included in a potential Hamilton cycle but then, indeed, vertex $x$ creates a problem. It gives us another trivial necessary condition: if $G_t$ has a Hamilton cycle, then there are no problematic vertices. Indeed, if $G_t$ has a vertex $v$ adjacent to three vertices, all of which are of degree 2, then $G_t$ cannot be Hamiltonian. This results in various types of ``problematic'' vertices. Our definition focuses only on a particular type for the purpose of simplifying the proof.

The numerical improvement is tiny and the bound we prove is certainly not tight. Hence, we only provide sketches of the proofs. The computations presented in the paper were performed by using Maple~\cite{Maple}. The worksheets can be found at the following address~\cite{Pawel}. 

\medskip

For convenience, we will distinguish a few phases in the semi-random graph process. The first phase lasts exactly $n \ln 2$ steps. Our first goal is to show that if the player plays greedily, then a.a.s.\ there will be linearly many problematic vertices at the end of first phase. 

\begin{claim}\label{claim:1}
Suppose that the player plays greedily during the first phase of the process. Then, a.a.s.\ there are $(\xi+o(1))n$ problematic vertices at the end of this phase, where
$$
\xi = \frac {1}{128} \left( 4(\ln 2)^4 + 20 (\ln 2)^3 + 54 (\ln 2)^2 - 18 \ln 2 - 21\right) \approx 0.0004035.
$$
\end{claim}

\begin{proof} 
It is fairly easy to show that the number of problematic vertices is a.a.s.\ at least $\xi n$ for some positive constant $\xi$. By the standard first and second moment calculations, after the first $(\ln 2/2)n$ steps there will be at least $(e^{-c} c^3/6) n$ vertices of in-degree at least 3 in $D_t$ where $c=\ln 2/2$. Then, a.a.s.\ a positive fraction of these vertices turns out problematic during the next $(\ln 2/2)n$ steps. Of course, in order to get larger constant $\xi$ it is best to track the process and apply the differential equation's method (see~\cite{DE} for more information on
the DE's method). We briefly sketch the argument.

For $a,b,c \in \{0, 1\}$ and $a \ge b \ge c$, we will say that a vertex $x$ in $D_t$ is of \textbf{type $(a,b,c)$} if it is of in-degree at least 3, with the first three in-neighbours $y_1, y_2$ and $y_3$ (order is not important), each of which has out-degree 1 and in-degree $a$, $b$, and $c$, respectively. In particular, vertex of type $(1,1,1)$ is simply a problematic vertex. Similarly, vertices of in-degree 2 could be of \textbf{type $(a,b)$} and vertices of in-degree 1 could be of \textbf{type $(a)$}. The remaining vertices of in-degree at least 1 are called \textbf{neglected}. (Note that neglected vertices can still prevent Hamilton cycle to be constructed but we simply neglect them.)

In order to analyze the process, we need to keep track of 9 random variables associated with vertices of different types, random variables $X_{abc}$, $X_{ab}$, and $X_{a}$. In particular, $X_{111}(t)$ is the number of problematic vertices (type $(1,1,1)$) at the end of step $t$. Moreover, let $Y(t)$ be the number of neglected vertices at the end of step $t$.  It is straightforward to compute the conditional expectations; for example, 
$$
\ex \Big( X_{111}(t+1)-X_{111}(t) ~~|~~ D_t \Big) = \frac {X_{110}(t)}{n} - 3 \ \frac{X_{111}(t)}{n}.
$$
Indeed, the only chance to create a problematic vertex is when the semi-random process selects the in-neighbour of a vertex of type $(1,1,0)$ that is of in-degree 0. On the other hand, if the process selects any of the first three in-neighbours of a problematic vertex, this vertex becomes neglected. The other expectations can be computed in a similar way. This suggests the following system of differential equations that should reflect the behaviour of the corresponding random variables:
\begin{eqnarray*}
x_{0}'(x) &=& 1 - x_{0}(x) - x_{00}(x) - x_{000}(x) - x_{1}(x) - x_{10}(x) - x_{100}(x) - x_{11}(x) \\
&&- x_{110}(x) - x_{111}(x) - y(x) - 2x_{0}(x), \\
x_{00}'(x) &=& x_{0}(x) - 3x_{00}(x), \\
x_{000}'(x) &=& x_{00}(x) - 3x_{000}(x), \\
x_{1}'(x) &=& x_{0}(x) - 2x_{1}(x), \\
x_{10}'(x) &=& 2x_{00}(x) + x_{1}(x) - 3x_{10}(x), \\
x_{100}'(x) &=& 3x_{000}(x) + x_{10}(x) - 3x_{100}(x), \\
x_{11}'(x) &=& x_{10}(x) - 3x_{11}(x), \\
x_{110}'(x) &=& 2x_{100}(x) + x_{11}(x) - 3x_{110}(x),\\ 
x_{111}'(x) &=& x_{110}(x) - 3x_{111}(x), \\
y'(x) &=& x_{1}(x) + x_{10}(x) + x_{100}(x) + 2x_{11}(x) + 2x_{110}(x) + 3x_{111}(x),
\end{eqnarray*}
with the initial condition that all functions at $x=0$ are equal to zero. This system of equations can be explicitly solved. In particular, we get that 
$$
x_{111}(x) = \frac{e^{-3x}x^4}{4} + \frac {5 e^{-3x} x^3}{4} + \frac {27 e^{-3x} x^2}{8} + \frac {39 e^{-3x}x}{8} + \frac {39 e^{-3x}}{16} - 3 e^{-2x} x - 3 e^{-2x} + \frac {9 e^{-x}}{16}.
$$
It follows from the DE's method that a.a.s.\ $X_{111}(t) = (1+o(1)) x_{111}(t/n) n$ for any $0 \le t \le n \ln 2$. Hence, a.a.s.\ the number of problematic vertices at the end of the first phase is equal to $(1+o(1)) x_{111}(\ln 2)$ and the claim holds.
\end{proof}

The above claim implies that if the player concentrates on achieving minimum degree 2 as soon as possible (that is, play greedily until the graph has minimum degree equal to 2), then a.a.s.\ there will be $(\xi+o(1)) n$ problematic vertices at the end of the first phase. If she continues playing greedily, then a.a.s.\ some positive fraction of these problematic vertices will remain present in the graph. Making them negligible will take linearly many steps. As a result, the player might want to adjust her strategy and not play greedily but start paying attention to problematic vertices instead. We now argue that this will also slow her down.

For a given $\delta \in [0,1]$ ($\delta=\delta(n)$ could be a function of $n$), let $\mathcal{F}_{\delta}$ be a family of strategies in which $(1-\delta) n \ln 2$ steps in the first phase are greedy (that is, the player selects some isolated vertex) but $\delta n \ln 2$ steps are non-greedy (that is, the player selects some vertex of degree at least 1).  We will show that playing non-greedily has a penalty in the form of reaching minimum degree 2 later in comparison to the minimum degree 2 process.

\begin{claim}\label{claim:2}
Fix any $\delta \in [0,1]$. For any strategy from family $\mathcal{F}_{\delta}$, a.a.s.\ it takes at least 
$$
(\ln 2 + \ln(1+\ln 2) + \eps_1(\delta) + o(1)) n 
$$
steps for $G_t$ to reach minimum degree 2, where 
$$
\eps_1(\delta) =
\ln \left( (2^{1 + \delta} - 1) \ln(2^{1 + \delta} - 1) - 2^{1 + \delta} \delta \ln 2 + (1 + \ln 2) 2^{\delta} \right) - \delta \ln 2  - \ln(1 + \ln 2),
$$ 
for $\delta\in[0,1/2]$ and $\eps_1(\delta)=\eps_1(1/2)$ for $\delta\in(1/2,1]$.
\end{claim}
Note that $\eps_1(\delta)$ is an increasing function of $\delta$ on $[0,1/2]$ and $\eps_1(0)=0$ (which corresponds to the original minimum degree 2 process).

\begin{proof} 
It is important to notice that the objective here is only to eliminate all vertices of degree below 2, and thus the player does not need to worry about problematic vertices. First consider $\delta\in[0,1/2]$.
As in the case of the unrestricted minimum degree 2 process (which corresponds to $\delta=0$), it is straightforward to see (for example, by a simple coupling argument) that it is always beneficial to play a greedy move instead of a non-greedy one\footnote{For any strategy $\bf{f}$ of $\scr{F}_{\delta}$ which does not prioritize greedy moves first, there exists another strategy within $\scr{F}_{\delta}$ which \textit{does} prioritize greedy moves first, and whose completion time is stochastically dominated by the completion time of $\bf{f}$.}. Hence, in order to achieve our goal, the best strategy from the family $\mathcal{F}_{\delta}$ is to play on vertices of degree 0 during the first $(1-\delta)n \ln 2$ steps. After that, the player should select vertices of degree 1 until the end of the first phase , that is, during the following $\delta n \ln 2$ steps. As there are no restrictions on the game after that (in particular, no restrictions on the number of non-greedy moves), she should play greedily until the end of the game; that is, play on vertices of degree 0 until they disappear and then play on vertices of degree 1 until the end of the game. Hence, both the first and the second phase are split into two sub-phases, depending on which type of vertices are selected.

In order to analyze how long it takes to finish this process, we need to keep track of two random variables: $Y(t)$ and $Z(t)$, the number of vertices at time $t$ of degree 0 and 1, respectively. We say that a move is of \textbf{type $i$} (where $i \in \{0,1\}$) if the player plays on a vertex of degree $i$. It is not difficult to see that 
\begin{eqnarray*}
\ex \Big( Y(t+1) - Y(t) ~~|~~ G_t \text{ and type } i \Big) &=& - \delta_{i = 0} - \frac {Y(t)}{n} \\
\ex \Big( Z(t+1) - Z(t) ~~|~~ G_t \text{ and type } i \Big) &=& \delta_{i = 0} - \delta_{i=1} + \frac {Y(t)}{n} - \frac {Z(t)}{n}. 
\end{eqnarray*}
where $\delta_A$ is the Kronecker delta function ($\delta_A=1$ if $A$ is true and $\delta_A=0$ otherwise). The corresponding system of DEs is
\begin{eqnarray*}
y'(x) &=& - \delta_{i = 0} - y(x) \\
z'(x) &=& \delta_{i = 0} - \delta_{i=1} + y(x) - z(x). 
\end{eqnarray*}
The initial condition is $y(0)=1$ and $z(0)=0$. Moreover, the final values of $y(x)$ and $z(x)$ after one of the sub-phases are used as the initial values for the next sub-phase. The conclusion follows from the DE's method. We skip the details and refer the interested reader to the Maple worksheets available on-line. 

It is easy to see that if $1/2<\delta\le 1$ then any strategy from ${\cal F}_{\delta}$ a.a.s.\ takes at least $(\ln 2+\ln(1+\ln 2) +\eps_1(1/2)+o(1))n$ steps to build a graph with minimum degree at least 2. During the second sub-phase of phase 1, the player may select any non-isolated vertex if there are no vertices of degree 1 left. These moves are not helping with building a graph with minimum degree 2 and thus it takes even longer to complete the process. 
\end{proof}

Our next task is to estimate the number of problematic vertices at the end of the first phase, provided that the player uses a strategy from family $\mathcal{F}_\delta$. 

\begin{claim}\label{claim:3}
Fix any $\delta \in [0,\xi/(2 \ln 2)]$, where $\xi$ is defined in Claim~\ref{claim:1}. For any strategy from family $\mathcal{F}_{\delta}$, a.a.s.\ there are at least $(\xi-2\delta \ln 2+o(1))n$ problematic vertices at the end of the first phase.
\end{claim}

\begin{proof}
It is not clear what the best strategy for minimizing the number of problematic vertices is. So, in order to keep the argument as simple as possible, we will help the player and propose to play the following auxiliary game, a mixture of on-line and off-line variants of the game. We simply run the greedy algorithm by selecting an isolated vertex in each step of the process. It follows from Claim~\ref{claim:1} that a.a.s.\ there are $(\xi+o(1))n$ problematic vertices at the end of the first phase. After that, we ask the player to `rewind' the process and carefully `rewire' $\delta$ fraction of moves in any way she wants keeping the remaining $1-\delta$ fraction of moves greedy, as required. Each modified move affects at most two problematic vertices so the number of problematic vertices decreases by at most $2 \cdot \delta n \ln 2$. Since this task clearly is much easier for the player than the original one, the lower bound follows. 
\end{proof}

Our final task is to combine all results together.

\begin{claim}\label{claim:4}
Fix any $\delta \in [0,\xi/(2 \ln 2)]$, where $\xi$ is defined in Claim~\ref{claim:1}. For any strategy from family $\mathcal{F}_{\delta}$, a.a.s.\ it takes at least 
$$
(\ln 2 + \ln(1+\ln 2) + \eps_1(\delta) + \eps_2(\delta) + o(1)) n 
$$
steps for $G_t$ to reach minimum degree 2 and remove all problematic vertices that were created during the first phase. Function $\eps_1(\delta)$ is defined in Claim~\ref{claim:2} and 
\begin{eqnarray*}
\eps_2(\delta) &=& \frac {\ln \big( 3 \tau(\delta) + 1 \big)}{3}, \\ 
\tau(\delta) &=& (\xi - 2\delta \ln 2)\ \exp(-3 \ln (1+\ln 2) - 3 \eps_1(\delta)). 
\end{eqnarray*}
\end{claim}

\begin{proof}
As in the proof of the previous claim, it is not clear what the best strategy is. Since we aim for an easy argument without optimizing the constants, we propose the player to play the following auxiliary game. We let her play the degree-greedy algorithm from the family $\mathcal{F}_\delta$ which optimizes the time needed to achieve minimum degree 2 (without worrying about problematic vertices). At the end of the first phase we artificially `destroy' some problematic vertices (if needed), leaving only $(\xi-2\delta \ln 2+o(1)) n$ of them in the graph. Clearly, this is an easier game for the player to play. Indeed, by Claim~\ref{claim:3} any strategy from $\mathcal{F}_\delta$ creates at least that many problematic vertices and so this is certainly a sweet deal for her. 

The player continues the game trying to reach minimum degree at least 2 and to destroy the remaining problematic vertices. It is straightforward to see that the best strategy is to continue playing the degree-greedy algorithm, destroying the remaining isolated vertices before playing vertices of degree 1. That part is taking $(\ln(1+\ln 2)+\eps_1(\delta)+o(1))n$ steps by Claim~\ref{claim:2}. In the meantime, vertices selected by the random graph process land on the neighbours of problematic vertices. The probability that a given problematic vertex is not destroyed is equal to 
$$
\left( 1 - \frac {3}{n} \right)^{(\ln(1+\ln 2)+\eps_1(\delta)+o(1))n} = \exp \Big( - 3 \big( \ln(1+\ln 2)+\eps_1(\delta) \big) \Big) + o(1).
$$
Hence a.a.s.\ there are $(\tau(\delta)+o(1))n$ problematic vertices at this point.

After that, the player has to destroy the remaining problematic vertices. Obviously, the best strategy is to choose $v_t$ to be one of the first three neighbours of a problematic vertex. A problematic vertex $x$ can also be destroyed if $u_t$ happens to be one of these neighbours. Let $Y(t)$ be the number of problematic vertices at the end of step $t$ (for simplicity counting from $t=0$). It is straightforward to see that 
$$
\ex \Big( Y(t+1) - Y(t) ~~|~~ G_t \Big) = - 1 - \frac {3Y(t)}{n}. 
$$
The corresponding DE is $y'(x) = - 1 - 3y(x)$ with the initial condition $y(0)=\tau(\delta)$. It follows that $y(x) = -1/3 + (\tau(\delta) +1/3)e^{-3x}$ and so we get that a.a.s.\ it takes another $(\eps_2(\delta)+o(1))n$ steps to finish the game, and the claim holds.
\end{proof}

Theorem~\ref{thm:lower_bound} follows immediately from Claim~\ref{claim:4}. Let us first extend $\eps_2(\delta)$ to $[0,1]$ by setting $\eps_2(\delta)=0$ for $\delta\in(\xi/(2\ln 2),1]$. We have shown that for every $\delta\in[0,1]$, any strategy from ${\cal F}_{\delta}$ a.a.s.\ takes at least $(\ln 2+\ln(1+\ln 2)+\eps_1(\delta)+\eps_2(\delta)+o(1))n$ steps to build a Hamilton cycle.  Note that $\eps_1(\delta)$ is an increasing function of $\delta$; the more non-greedy moves the player needs to play, the longer the game is. On the other hand, $\eps_2(\delta)$ is a decreasing function on $[0,\xi/(2 \ln 2)]$ with $\eps_2(\xi/(2 \ln 2))=0$; the non-greedy moves can be spent on destroying problematic vertices and so the number of them decreases with $\delta$. After more careful investigation we get that $\eps_1(\delta) + \eps_2(\delta)$ is a decreasing function on $[0,\xi/(2 \ln 2)]$ and then it is equal to $\eps_1(\delta)$ and so it starts increasing. Therefore we get that
$$
\eps = \min_{\delta} \Big( \eps_1(\delta) + \eps_2(\delta) \Big) = \eps_1 \left( \frac {\xi}{2 \ln 2} \right) + \eps_2 \left( \frac {\xi}{2 \ln 2} \right) = \eps_1 \left( \frac {\xi}{2 \ln 2} \right) \approx 2.403 \cdot 10^{-8}.
$$


\begin{thebibliography}{99}

\bibitem{process2} O.\ Ben-Eliezer, L.\ Gishboliner, D.\ Hefetz, M.\ Krivelevich, Very fast construction of bounded degree spanning graphs via the semi-random graph process.  Proceedings of the 31st Symposium on Discrete Algorithms (SODA'20), 728--737 (2020).

\bibitem{process1} O.\ Ben-Eliezer, D.\ Hefetz, G.\ Kronenberg, O.\ Parczyk, C.\ Shikhelman, M.\ Stojakovic, Semi-random graph process, to appear in \emph{Random Structures \& Algorithms}.

\bibitem{Julia} J.\ Bezanson, A.\ Edelman, S.\ Karpinski, V.B.\ Shah, Julia: A Fresh Approach to Numerical Computing, \emph{SIAM Review}, 59(1): 65--98 (2017). 

\bibitem{3-out} T.\ Bohman, A.M.\ Frieze, Hamilton cycles in 3-out, \emph{Random Structures \& Algorithms} 35(4): 393--417 (2009).

\bibitem{Jump} I.\ Dunning, J.\ Huchette, Miles\ Lubin, JuMP: A Modeling Language for Mathematical Optimization, \emph{SIAM Review}, 59(2):295--320 (2017).

\bibitem{lll}
B. V. Gnedenko, On the local limit theorem in the theory of probability, {\em Uspekhi Mat. Nauk}, 3:187--194, 1948.


\bibitem{Multistart} R.\ Martí, Multi-Start Methods. In: F.\ Glover, G.A.\ Kochenberger (eds) Handbook of Metaheuristics. International Series in Operations Research \& Management Science, vol 57. Springer (2003).

\bibitem{Maple} M.B.\ Monagan, K.O.\ Geddes, K.M.\ Heal, G.\ Labahn, S.M.\ Vorkoetter, J.\ McCarron, and P.\ DeMarco, Maple 10 Programming Guide, Maplesoft, Waterloo ON, Canada, 2005.

\bibitem{Posa} L.~Pos\'{a}, Hamiltonian circuits in random graphs, \emph{Discrete Mathematics} 14: 359--364 (1976).

\bibitem{Shrijver} A.\ Schrijver, Combinatorial optimization: polyhedra and efficiency, Springer (2003).

\bibitem{Ipopt}  A.\ Wächter, L.T.\ Biegler, On the Implementation of a Primal-Dual Interior Point Filter Line Search Algorithm for Large-Scale Nonlinear Programming, \emph{Mathematical Programming} 106 25--57  (2006).

\bibitem{DE} N.C.\ Wormald, The differential equation method for random graph processes and greedy algorithms. Lectures on Approximation and Randomized Algorithms, eds.\ M.\ Karo\'nski and H.J.\ Pr\"{o}mel, PWN, Warsaw, pp.\ 73--155, 1999.

\bibitem{Pawel} \texttt{https://math.ryerson.ca/$\sim$pralat/}

\end{thebibliography}
\end{document}